\documentclass[10.5pt,reqno]{amsart}
\usepackage{longtable}
\usepackage{hyperref}
\usepackage[T1]{fontenc}
\usepackage[utf8]{inputenc}
\usepackage[english]{babel}
\usepackage{textcomp}
\usepackage{dsfont}
\usepackage{latexsym}
\usepackage{amssymb}
\usepackage{amsthm}
\usepackage{amsmath}
\DeclareMathAlphabet{\mathpzc}{OT1}{pzc}{m}{en}
\usepackage{yfonts}
\usepackage{xfrac}
\usepackage{newlfont}
\usepackage{graphicx}
\usepackage{mathtools}
\usepackage{comment}
\usepackage{indentfirst}
\usepackage{braket}
\usepackage{mathrsfs}
\usepackage{fixmath}

\usepackage{scalerel}[2014/03/10]
\usepackage[usestackEOL]{stackengine}
\newcommand{\dashint}{\,\ThisStyle{\ensurestackMath{%
			\stackinset{c}{.2\LMpt}{c}{.5\LMpt}{\SavedStyle-}{\SavedStyle\phantom{\int}}}%
		\setbox0=\hbox{$\SavedStyle\int\,$}\kern-\wd0}\int}

\DeclareMathOperator{\supp}{Supp}

\DeclareMathOperator{\Aut}{Aut}

\newcommand{\Aff}{\mathrm{Aff}}
\newcommand{\Hol}{\mathrm{Hol}}

\newcommand{\Supp}[1]{\supp\left( #1\right) }

\newcommand{\ee}{\mathrm{e}}

\newcommand{\vect}[1]{\mathbf{{#1}}}
\newcommand{\dd}{\mathrm{d}}

\DeclarePairedDelimiter{\abs}{\lvert}{\rvert}

\DeclarePairedDelimiter{\norm}{\lVert}{\rVert}

\let\originalleft\left
\let\originalright\right
\renewcommand{\left}{\mathopen{}\mathclose\bgroup\originalleft}
\renewcommand{\right}{\aftergroup\egroup\originalright}

\newcommand{\N}{\mathds{N}}
\newcommand{\Z}{\mathds{Z}}

\newcommand{\C}{\mathds{C}}
\newcommand{\D}{\mathds{D}}
\newcommand{\Hb}{\mathds{H}}

\newcommand{\R}{\mathds{R}}

\newcommand{\T}{\mathds{T}}

\newcommand{\Us}{\mathscr{U}}

\newcommand{\Ac}{\mathcal{A}}
\newcommand{\Bc}{\mathcal{B}}
\newcommand{\Cc}{\mathcal{C}}
\newcommand{\Dc}{\mathcal{D}}
\newcommand{\Ec}{\mathcal{E}}
\newcommand{\Fc}{\mathcal{F}}
\newcommand{\Gc}{\mathcal{G}}

\newcommand{\Ic}{\mathcal{I}}

\newcommand{\Kc}{\mathcal{K}}

\newcommand{\cM}{\mathcal{M}}
\newcommand{\Nc}{\mathcal{N}}

\newcommand{\Pc}{\mathcal{P}}

\newcommand{\Sc}{\mathcal{S}}

\newcommand{\Uc}{\mathcal{U}}

\renewcommand{\Im}{\mathrm{Im}\,}
\renewcommand{\Re}{\mathrm{Re}\,}

\newcommand{\meg}{\leqslant}
\newcommand{\Meg}{\geqslant}
\newcommand{\eps}{\varepsilon}
\renewcommand{\phi}{\varphi}
\newcommand{\mi}{\mu}

\newcommand{\leftexp}[2]{{\vphantom{#2}}^{#1}{#2}} 
\newcommand{\trasp}{\leftexp{t}}
\newcommand{\Lin}{\mathscr{L}}

\title[Invariant Spaces of Holomorphic Functions]{Invariant Spaces of
  Holomorphic Functions on the Siegel Upper Half-Space}

\date{}

\begin{document}

\theoremstyle{definition}
\newtheorem{deff}{Definition}[section]

\newtheorem{oss}[deff]{Remark}

\newtheorem{ass}[deff]{Assumptions}

\newtheorem{nott}[deff]{Notation}

\theoremstyle{plain}
\newtheorem{teo}[deff]{Theorem}

\newtheorem*{teo*}{Theorem}

\newtheorem{lem}[deff]{Lemma}

\newtheorem{prop}[deff]{Proposition}

\newtheorem{cor}[deff]{Corollary}

\author[M. Calzi, M. M. Peloso]{Mattia Calzi, 
	Marco M. Peloso}

\address{Dipartimento di Matematica, Universit\`a degli Studi di
	Milano, Via C. Saldini 50, 20133 Milano, Italy}
\email{{\tt mattia.calzi@unimi.it}}
\email{{\tt marco.peloso@unimi.it}}

\thanks{{\em Math Subject Classification 2020}: 46E15, 47B33, 32M15 }
\thanks{The authors are members of the 	Gruppo Nazionale per l'Analisi
  Matematica, la Probabilit\`a e le	loro Applicazioni (GNAMPA) of
  the Istituto Nazionale di Alta Matematica (INdAM). The authors were partially funded by the INdAM-GNAMPA Project CUP\_E55F22000270001.
} 

\begin{abstract}
	 In this paper we consider the (ray) representations of the group $\Aut$ of biholomorphisms of the Siegel upper half-space $\Uc$ defined by $U_s(\phi) f=(f\circ \phi^{-1}) (J \phi^{-1})^{s/2}$, $s\in\R$, and characterize the semi-Hilbert spaces $H$ of holomorphic functions on $\Uc$ satisfying the following assumptions:
	\begin{enumerate}
		\item[(a)] $H$ is strongly decent;
		
		\item[(b)] $U_s$ induces a bounded ray representation of the group $\Aff$ of affine automorphisms of $\Uc$ in $H$.
	\end{enumerate}
	We use this description to improve the known characterization of the semi-Hilbert spaces of holomorphic functions on $\Uc$ satisfying (a) and (b) with $\Aff$ replaced by $\Aut$.
	
	In addition, we characterize the mean-periodic holomorphic functions on $\Uc$ under the representation $U_0$ of $\Aff$.
\end{abstract}
\maketitle

\section{Introduction}

Let $\D$ be the unit disc in $\C$, and consider the (M\"obius) group of its biholomorphisms, namely
\begin{equation}\label{def:AutD}
\Aut(\D)\coloneqq \Set{z\mapsto \alpha\frac{z-b}{1-z\overline b}\colon \alpha\in \T, \abs{b}<1}.
\end{equation}
 Many classical spaces of holomorphic functions on $\D$ enjoy the
property of being 
preserved by composition with elements of $\Aut(\D)$. Notable examples include the space $H^\infty(\D)$ of bounded holomorphic functions on $\D$, the Bloch space
\[
\Bc(\D)=\Set{f\in \Hol(\D)\colon \sup_{z\in \D} (1-\abs{z}^2) \abs{f'(z)}<\infty},
\]
and the Dirichlet space
\[
\Dc(\D)=\Set{f\in \Hol(\D)\colon \int_\D \abs{f'(z)}^2\,\dd z<\infty}.
\]
In these examples, not only the considered space is invariant under composition by the elements of $\Aut(\D)$, but also their natural seminorms are: namely, the $\sup$-norm for $H^\infty(\D)$, the seminorm 
\[
f\mapsto \sup_{z\in \D} (1-\abs{z}^2) \abs{f'(z)}
\]
for the Bloch space $\Bc(\D)$, and the seminorm
\[
f\mapsto \int_\D \abs{f'(z)}^2\,\dd z
\]
for the Dirichlet space $\Dc(\D)$. These spaces are usually said to be M\"obius-invariant in the literature, even though there seems to be no general agreement about which axioms  a M\"obius-invariant space should satisfy. 
For example, in~\cite{ArazyFisher,ArazyFisherPeetre,Peloso} a semi-Banach space\footnote{A vector space endowed with a seminorm with respect to which it is complete.} $X$ of holomorphic functions on $\D$ is said to be M\"obius-invariant if the following hold:
\begin{enumerate}
	\item[(1)] $X$ embeds continuously in $\Bc(\D)$ (endowed with the  non-Hausdorff topology defined above);
	
	\item[(2)] there is a constant $C\Meg 1$ such that for every $f\in X$ and for every $\phi\in \Aut(\D)$, the function $f\circ \phi$ belongs to $X$, and $\norm{f\circ \phi}_X\meg C\norm{f}_X$;
	
	\item[(3)] for every $f\in X$, the mapping $\Aut(\D)\ni \phi\mapsto f\circ \phi\in X$ is continuous.
\end{enumerate}
Condition (2) is what properly states that $X$ is M\"obius-invariant, while condition (1) is imposed to prevent pathological spaces, and condition (3) is only needed to simplifying some arguments, but may be weakened without compromising the resulting theory.
In fact, condition (3) is superfluous when the Banach space associated
with $X$ (that is, $X$ if $X$ is Hausdorff and  $X/\C$ otherwise) is
reflexive (cf.~\cite[Corollary 4.4]{Survey});  moreover it prevents from
considering some natural spaces, such as $H^\infty(\D)$ and
$\Bc(\D)$.  Furthermore,  
no condition is imposed to ensure that the constant functions belong
to $X$, if $X\neq \Set{0}$, even though this fact is tacitly assumed
in~\cite{ArazyFisher,ArazyFisherPeetre,Peloso}. We shall discuss this
and other related issues in Section~\ref{sec:3}. 

Conditions (2) and (3) may then be rephrased saying that the mapping 
\[
U_0\colon \Aut(\D)\ni \phi\mapsto [f\mapsto f\circ \phi^{-1}]\in \Lin(\Hol(\D))
\]
induces a \emph{continuous bounded} representation of $\Aut(\D)$ in $X$.

We also remark that~\cite{RubelTimoney} shows that, assuming condition
(2) to hold,  condition (1) is equivalent to assuming that either
$X=\Set{0}$ or $X$ admits a decent continuous linear functional, that
is, a continuous linear functional which is continuous for the
topology of compact convergence on $\D$ (or, equivalently, which
extends to a continuous linear functional on $\Hol(\D)$). Thus, this
condition may be seen as a rather natural replacement for the
condition that $X$ should embed continuously in $\Hol(\D)$, which
nonetheless is only available when $X$ is Hausdorff. These conditions
were later developed in~\cite{ArazyFisher2,Arazy} in order to deal
with more general contexts, but remain essentially analogous. 

More recently, in~\cite{AlemanMas} a different notion of M\"obius
invariant spaces was introduced (actually,~\cite{AlemanMas} deals with
some suitable weighted action of $\Aut(\D)$ on $\Hol(\D)$ which we
shall describe below, but we shall translate their definitions in our
context for the sake of simplicity). First of all, because of the
slightly different context considered in~\cite{AlemanMas}, only Banach
spaces are considered. Then, a `conformally invariant space of index
$0$', according to~\cite{AlemanMas}, is a Banach space $X$ of
holomorphic functions such that the following hold: 
\begin{enumerate}
	\item[(1$'$)] $X$ embeds continuously in $\Hol(\D)$;
	
	\item[(2)] there is a constant $C\Meg 1$ such that for every
          $f\in X$ and for every $\phi\in \Aut(\D)$, the function
          $f\circ \phi$ belongs to $X$, and $\norm{f\circ \phi}_X\meg
          C\norm{f}_X$; 
	
	\item[(3$'$)] for every $R>1$, $\Hol(R\D)\subseteq X$.
\end{enumerate}
We observe explicitly that (1$'$) is the natural replacement of (1)
for Banach spaces. Notice that, in the context actually considered
in~\cite{AlemanMas}, this requirement is perfectly natural: reasonably
extending the above definition to semi-Banach spaces would \emph{not}
provide any further examples as a consequence of Remark~\ref{oss:1}.  As for
what concerns condition (3$'$), it is essentially weaker than
condition (3): if condition (3) holds, then $f(R\,\cdot\,)\in X$ for
every $f\in X$ and for every $R\in (0,1)$ (see the proof of~\cite[Proposition 1]{ArazyFisher2} or~\cite[Proposition 4.3]{Survey}). In addition, if $X$
contains both non-trivial constant and non-constant functions, then
$X$ is dense in $\Hol(\D)$, so that we are not too far away from the
requirement $\Hol(R\D)\subseteq X$ for every $R>1$. \medskip

In this paper, we shall follow a somewhat different strategy. Indeed, we shall drop both conditions (3) and (3$'$), as the former will be automatically satisfied (cf.~Proposition~\ref{prop:32}), while the latter will not always be satisfied, and does not seem to be of any use for our purposes. Since we shall also consider some weaker versions of (2), which we will discuss later on, we shall replace (1) with the requirement that $X$ should be \emph{strongly} decent. This condition means that the space of continuous linear functionals on $X$ which extend to continuous linear functionals on $\Hol(\D)$ is  dense in the weak dual topology of the dual $X'$ of $X$. Assuming $X$ to be a semi-Banach space, this assumption is equivalent to the existence of a closed vector subspace $V$ of $\Hol(\D)$ such that $X\cap V$   is the closure of $\Set{0}$ in $X$, and such that the canonical mapping $X\to \Hol(\D)/V$ is continuous.

Proceeding further into this matter, we observe that there are several other Banach spaces of holomoprhic functions on $\D$ on which $\Aut(\D)$ has a natural isometric `weighted action'. For example, set, for every $s\in \R$ and for every $\phi\in \Aut(\D)$,
\[
\widehat U_s(\phi) f= (f\circ \phi^{-1}) (J\phi^{-1})^{s/2}
\]
for every $f\in \Hol(\D)$, where $J\phi^{-1}$ denotes the complex
Jacobian of $\phi^{-1}$. Observe that, since $\D$ is convex, the power
$(J\phi^{-1})^{s/2}$ may be defined as a holomorphic function on
$\D$. Nonetheless, unless $s\in 2 \Z$, this power is only defined up
to the multiplication by suitable unimodular constants, so that we
only define $\widehat U_s(\phi)$ as an element of $
\Lin(\Hol(\D))/\T$. In other words, $\widehat U_s $ becomes a
\emph{ray} representation (cf.~\cite{Bargmann}) of $\Aut(\D)$ in
$\Hol(\D)$ instead of an ordinary one. It then makes sense to say that
$\widehat U_s$ induces a bounded or isometric \emph{ray}
representation in any semi-Banach space of holomorphic functions on
$\D$, but one cannot define integral operators of the form $\widehat
U_s(f)$ for $f\in L^1(\Aut(\D))$, even if the \emph{ray}
representation is continuous. In order to circumvent this difficulty,
one may lift $\widehat U_s$ to a continuous ordinary representation
$\widetilde U_s$ of the universal covering group $\widetilde \Aut(\D)$
in $\Hol(\D)$, as we shall do in the sequel. 

Then, for every $p\in [1,\infty)$ and for every $s>2/p$, the weighted Bergman spaces\footnote{  We remark explicitly that this parametrization of the weighted Bergman spaces is different from the one we shall use in the sequel.}
\[
A^p_s(\D)\coloneqq\Set{f\in \Hol(\D)\colon \int_{\D} \abs{f(z)}^p (1-\abs{z}^2)^{ps/2-2}\,\dd z }
\]
are invariant under the action of $\widetilde U_s$  with their norms (cf.~Proposition~\ref{prop:24}). Notable examples arise when $p=2$, in which case $A^2_s(\D)$ is a reproducing kernel Hilbert space (RKHS for short) of holomorphic functions on $\D$, that is, a Hilbert space which embeds continuously in $\Hol(\D)$. More precisely, its reproducing kernel is 
\[
\Kc_{s}\colon(z,z')\mapsto c_s (1-z\overline {z'})^{-s}
\]
for a suitable $c_s>0$; this means that $\Kc_s(\,\cdot\,,z)\in A^2_s(\D)$ and that
\[
f(z)=\langle f\vert \Kc_s(\,\cdot\,, z)\rangle_{A^2_s(\D)}
\]
for every $f\in A^2_s(\D)$ and for every $z\in \D$.

From the viewpoint of the representation theory of the simple Lie
group $\Aut(\D)$ (and of its covering group $\widetilde \Aut(\D)$),
the representations $\widetilde U_s$ of $\widetilde \Aut(\D)$ in
$A^2_s(\D)$, for $s>1$, are of particular importance, as they may be
identified with the  discrete series representations of $\widetilde
\Aut(\D)$ (cf.~\cite{VergneRossi}). In other words, every irreducible
subrepresentation of the left regular representation of $\widetilde
\Aut(\D)$ in $L^2(\widetilde \Aut(\D))$ is unitarily equivalent to one
of the representations $\widetilde U_s$ in $A^2_s(\D)$, $s>1$. 

Now, observe that saying that $A^2_s(\D)$ is $\widetilde U_s$-invariant with its norm is equivalent to saying that $[\widetilde U_s(\phi)\otimes \overline{\widetilde U_s(\phi)} ]\Kc_s=\Kc_s$ for every $\phi \in \widetilde \Aut(\D)$. In particular, the sesquiholomorphic function
\[
\Kc^{ s}\colon (z,z')\mapsto (1-z\overline{z'})^{-s}
\]
satisfies $[\widetilde U_s(\phi)\otimes \overline{\widetilde U_s(\phi)}] \Kc^s=\Kc^s $ for every $s\in \R$ and for every $\phi \in \widetilde \Aut(\D)$. As a consequence, if $\Kc^s$ is the reproducing kernel of some RKHS $\Ac_s$, then $\Ac_s$ is $\widetilde U_s$-invariant with its norm. This is the case if and only if $\Kc^s$ satisfies a suitable positivity condition, namely
\[
\sum_{j,k=1}^N \alpha_j \overline{\alpha_k} \Kc^s(z_j, z_k)\Meg 0
\]
for every $N\in\N$, for every $z_1,\dots, z_N\in \D$, and for every $\alpha_1,\dots, \alpha_N\in \C$, and this happens exactly when $s \Meg 0$ (cf.~\cite{VergneRossi}). In this case, $\Ac_s$ is the completion of the vector subspace of $\Hol(\D)$ generated by the $\Kc^s(\,\cdot\, ,z)$, as $z$ runs through $\D$, endowed with the unique scalar product such that
\[
\langle \Kc^s(\,\cdot\,,z)\vert \Kc^s(\,\cdot\,,z')\rangle_{\Ac_s}=\Kc^s(z,z')
\]
for every $z,z'\in \D$. Observe that $\Ac_s$ embeds canonically and continuously into $\Hol(\D)$, so that we may consider it as a vector subspace of $\Hol(\D)$, endowed with a finer topology.
For example, $\Ac_1$ is the Hardy space
\[
H^2(\D)=\Set{f\in \Hol(\D)\colon \sup_{r\in (0,1)} \int_\T \abs{f(rz)}^2\,\dd z<\infty },
\]
endowed with a proportional  norm, while $\Ac_0$ is simply the space of constant functions.
It is interesting to observe that also a converse result holds: namely, if $H$ is a non-trivial RKHS of holomorphic functions on $\D$ and $H$ is $\widetilde U_s$-invariant with its norm, then $s\Meg 0$ and $H=\Ac_s$ with a proportional norm (cf.~\cite{ArazyFisher3}). 
If we momentarily ignore the problem of providing simpler descriptions for the spaces $\Ac_s$, $s\in (0,1)$, these results solve the problem of classifying $\widetilde U_s$-invariant RKHS of holomorphic functions on $\D$. In particular, this tells us that the Dirichlet space $\Dc(\D)$, which is $\widetilde U_0$-invariant with its natural \emph{seminorm}, may \emph{not} be endowed with any finer $\widetilde U_0$-invariant norm for which $\Dc(\D)$ embeds continuously in $\Hol(\D)$.
This naturally leads to the problem of finding and describing \emph{semi-Hilbert} spaces of holomorphic functions on $\D$ which are   $\widetilde U_s$-invariant with their seminorms, and satisfy some additional assumptions such as (1) above (for $s=0$) or the strong decency condition we described earlier. 
To this end, we first observe that, if $s\in-\N$, then the $(1-s)$-th order derivative intertwines $\widetilde U_s$ and $\widetilde U_{2-s}$, that is,
\[
(\widetilde U_s (\phi)f)^{(1-s)}=\widetilde U_{2-s} (\phi)f^{(1-s)}
\]
for every $f\in \Hol(\D)$ and for every $\phi\in \widetilde \Aut(\D)$ (cf.~\cite[Theorem 6.4]{Arazy}). Consequently, the semi-Hilbert space
\[
H_s=\Set{f\in \Hol(\D)\colon f^{(1-s)}\in \Ac_{2-s}}
\]
is $\widetilde U_s$-invariant with its seminorm $f\mapsto \norm{f^{(1-s)}}_{\Ac_{2-s}}$   for every $s\in -\N$. Conversely, if $H$ is a non-Hausdorff non-trivial decent semi-Hilbert space of holomorphic functions which is $\widetilde U_s$-invariant with its seminorm, then  $s\in -\N$ and $H$ embeds as a \emph{dense} subspace of $H_s$, with a proportional seminorm (cf.~\cite{ArazyFisher,ArazyFisher2,Arazy} and Theorem~\ref{teo:6}). Notice that, since $H_s$ is \emph{not} Hausdorff, it admits non-closed complete dense vector  subspaces, which are precisely the vector subspaces $V$ such that $V+\overline{\Set{0}}^{H_s}=H_s$, where $\overline{\Set{0}}^{H_s}$ denotes the closure of $\Set{0}$ in $H_s$. We observe explicitly that we do \emph{not} know whether $H_s$ contains \emph{proper} complete dense $\widetilde U_s$-invariant vector subspaces.
\medskip

These  problems have been investigated also in different contexts and generality. On the one hand, one may consider invariant Banach (or semi-Banach) spaces of holomorphic functions instead of simply Hilbert (or semi-Hilbert) spaces. As one may expect, in this case one cannot find a complete classification, but one may still find minimal and maximal spaces and study some general properties of invariant spaces. Cf.~\cite{ArazyFisher4,ArazyFisherPeetre,Zhu,Peloso,ArazyUpmeier3,AlemanMas,Survey} for more on this problem.  
On the other hand, one may replace the unit disc $\D$ with more general domains, such as the unit ball in $\C^n$ (cf.~\cite{ArazyFisher2,Zhu,Peloso,Arazy3}), or more general irreducible symmetric domains (cf.~\cite{FarautKoranyi2,ArazyFisher2,Arazy2,Arazy,ArazyUpmeier4,ArazyUpmeier,Yan,ArazyUpmeier2,Tubi,Survey}).

In this paper we consider a related problem. Our first aim was to provide a reasonable description of invariant semi-Hilbert spaces on the unbounded realizations of irreducible symmetric domains as Siegel domains of type II. For example, the natural Siegel domain associated with the unit disc $\D$ is the upper half plane $\C_+=\R+i \R_+^*$. Since we noticed that many of these spaces could be essentially characterized requiring only the invariance under the action of the group of \emph{affine} biholomorphisms, we decided to perform a more detailed study of \emph{affinely} invariant semi-Hilbert spaces of holomorphic functions on symmetric Siegel domains. In this paper we shall consider only rank $1$ spaces, namely the Siegel upper half-spaces
\[
\Uc_{n+1}\coloneqq\Set{(\zeta,z)\in \C^{n+1}\colon \Im z>\abs{\zeta}^2},
\]
so that $\Uc_{n+1}$ is biholomorphic to the unit ball in
$\C^{n+1}$. Notice that this domain reduces to the upper half-plane
when $n=0$. The first author will deal with more general domains in a forthcoming paper~\cite{Tubi}. 

Notice that the group $\Aff(\Uc_{n+1})$ of affine biholomorphisms  of $\Uc_{n+1}$ is particularly simple to describe  (cf.~Lemma~\ref{lem:2.1}). In addition, the representation $\widetilde U_s$ essentially induces the representation $U_s$ of $\Aff(\Uc_{n+1})$ defined by\footnote{We shall comment on the exponent $s/(n+2)$ later on.}
\[
U_s(\phi)f= (f\circ \phi^{-1}) \abs{J \phi^{-1}(0,i)}^{s/(n+2)},
\]
in the sense that $U_s(\phi)$ differs from the various $\widetilde U_s(\psi)$, for $\psi\in \widetilde \Aut(\Uc_{n+1})$ acting as $\phi$ on $\Uc_{n+1}$, by a unimodular constant.

As it turns out, the decency assumption, which was essentially (though not completely) sufficient in order to deal with $\widetilde U_s$-invariant spaces, no longer seems to be sufficient for the study of $U_s$-invariant spaces. This is the reason that brought us to develop the notion of strongly decent spaces.

We shall therefore consider the following problem: `provide a
description of the strongly decent semi-Hilbert spaces in which $U_s$
induces  a bounded representation of $\Aff(\Uc_{n+1})$'. In fact,
we are able to provide a reasonably complete classification of these
spaces, see~Theorem~\ref{teo:6} for the case $n=0$ and
Theorem~\ref{teo:3} for the case $n\ge1$. 
Explicitly, when $n=0$, we shall prove the following result.

\begin{teo*}
	Take $s\in \R$ and a strongly decent semi-Hilbert space $H$ of holomorphic functions on $\C_+$. Assume that $U_s$ induces a bounded (resp.\ isometric) representation of $\Aff(\C_+)$ in $H$. Then, there is $k\in\N$ such that $s+2k\Meg 0$ and such that $H$ embeds as a dense subspace of
	\[
	\Set{f\in \Hol(\C_+)\colon f^{(k)}\in \Ac_{s+2 k}(\C_+)}
	\]
	with an equivalent (resp.\ proportional) seminorm.
\end{teo*}

Here,  $\Ac_{s+2k}(\C_+)$ denotes the  RKHS of holomorphic functions on $\C_+$ whose reproducing kernel is $(z,z')\mapsto [(z-\overline{z'})/(2i)]^{-s-2k}$. Notice that $\Ac_{s+2k}(\C_+)$ is $U_{s+2 k}$-invariant with its norm,  so that it is canonically isomorphic to the space $\Ac_{s+2k}$ described before. 
The case of semi-Hilbert spaces of holomorphic functions on $\Uc_{n+1}$ is similar, but somewhat more complicated.

We shall then use the classification of $U_s$-invariant spaces to reobtain and slightly improve the classification of $\widetilde U_s$-invariant spaces.

Our techniques essentially rely on two ingredients: (1) a complete description of the \emph{closed} $U_s$-invariant (or, equivalently, $U_0$-invariant) vector subspaces of $\Hol(\Uc_{n+1})$; (2) the intertwining properties of the derivatives $\partial_z^k$ between the representations $U_s$ and $U_{s+2k}$.

Concerning (1), we shall first show that every closed $U_0$-invariant vector subspace $V$ of $\Hol(\Uc_{n+1})$ is the closure of the set of polynomials contined in $V$ (cf.~Proposition~\ref{prop:6}). This result may be considered as a complete description of mean-periodic functions in $\Hol(\Uc_{n+1})$ (with respect to the action of $\Aff(\Uc_{n+1})$ induced by $U_0$). We shall then provide a description of the $U_0$-invariant vector spaces of holomorphic polynomials on $\Hol(\Uc_{n+1})$.
This description, combined with the strong decency condition on $H$, will show that  $H$ is determined by $\partial^k_2 H$ for some $k\in\N$, in the sense that $\partial_2^k H$ is a reproducing kernel Hilbert space of holomorphic functions on $\Uc_{n+1}$ and $H\cap \ker \partial_2^k$ is the closure of $\Set{0}$ in $H$. Notice that this is rigorously so when $n=0$, while in the general case the situation is slightly more complicated. Then, (2) will imply that $\partial_2^k H$ is actually $U_{s+2k}$-invariant, so that by a simple comparison between the reproducing kernels of $\partial_2^k H$ and $\Ac_{s+2 k}(\C_+)$ (when $s+2k \meg 0$), using the fact that $\Aff(\Uc_{n+1})$ is amenable and acts transitively on $\Uc_{n+1}$, we see that $\partial_2^k H=\Ac_{s +2k}(\C_+)$. When $n=0$, this provides the desired description of the possible $H$, whereas in the general case some further arguments are necessary.

Here is a plan of the paper. In Section~\ref{sec:2}, we shall review some basic facts on the group $\Aut(\Uc_{n+1})$ of biholomorphisms of $\Uc_{n+1}$ and some of its notable subgroups. We shall then review some weighted Bergman and related spaces, and prove their $\widetilde U_s$-invariance. Finally, we shall discuss some properties of (strongly) decent (and saturated) spaces in a general framework. 

In Section~\ref{sec:3}, we shall characterize the $U_s$-invariant reproducing kernel Hilbert spaces of holomorphic functions on $\Uc_{n+1}$ (cf.~Theorem~\ref{teo:1}), and discuss the unitary equivalences between the associated unitary representations of $\Aff(\Uc_{n+1})$ (and $\widetilde \Aut(\Uc_{n+1})$, whenevery possible), cf.~Remark~\ref{oss:2}.

In Section~\ref{sec:4bis}, we shall characterize the strongly decent $U_s$-invariant semi-Hilbert spaces of holomorphic functions on $\Uc_0=\C_+$ (cf.~Theorem~\ref{teo:6}), and discuss the unitary equivalences between the associated unitary representations of $\Aff(\Uc_{n+1})$ (and $\widetilde \Aut(\Uc_{n+1})$, whenevery possible), cf.~Remark~\ref{oss:4}.
In Section~\ref{sec:5}, we shall perform a similar study in the general case (cf.~Theorems~\ref{teo:3} and~\ref{teo:4}, and Remark~\ref{oss:3}).

In Section~\ref{sec:4}, we shall trasfer the preceding results to the bounded setting, thus studying spaces of holomorphic functions on the unit ball $B_{n+1}$ of $\C^{n+1}$ (cf.~Proposition~\ref{prop:39}), as well as compare our strong decency (and saturation) conditions with the ones appearing in the literature (cf.~Proposition~\ref{prop:38}). Even though the results on $\widetilde U_s$-invariant semi-Hilbert spaces of holomorphic functions on $B_{n+1}$ available in the literature are fairly complete, the characterization of the $\widetilde U_s$-invariant semi-Hilbert spaces $H$ of holomorphic functions on $B_{n+1}$ for $s<0$ and $n>0$ under the sole assumption that $\widetilde U_s$ induces a bounded (and not necessarily isometric) representation in $H$ which stems from Theorem~\ref{teo:4} seems to be new.

Finally, in Section~\ref{sec:app} we shall prove some general results on the closed $\Aff(\Uc_{n+1})$-$U_0$-invariant vector subspaces of $\Uc_{n+1}$. Since these facts are of independent interest and may be proved for general Siegel domains (of type II) with no additional effort, we shall consider the general case instead of confining ourselves to the domains $\Uc_{n+1}$.

\section{Preliminaries}\label{sec:2}

We denote by $B_{n+1}$ the unit ball in $\C^n\times \C$, and by
$\Uc_{n+1}$ the Siegel upper half-space 
\[
\Uc_{n+1}\coloneqq \Set{(\zeta,z)\in \C^n\times \C\colon \Im z>\abs{\zeta}^2}.
\] 
Notice that, when $n=0$, $B_1$ is the unit disc in $\C$ and $\Uc_1$ is the
upper half-plane   $\C_+$.   
 Then,   the Cayley transform
\[
\Cc\colon B_{n+1}\ni (\zeta,z)\mapsto \left( \frac{\zeta}{1-z}, i \frac{1+z}{1-z} \right)\in \Uc_{n+1},
\]
with inverse
\[
\Cc^{-1} \colon \Uc_{n+1}\ni (\zeta,z)\mapsto \left(2i \frac{\zeta}{z+i} , \frac{z-i}{z+i}\right)\in B_{n+1},
\]
induces a biholomorphism between $B_{n+1}$ and $\Uc_{n+1}$.

\subsection{Groups of Automorphisms}\label{sec:2:1}
 Given a domain $D$, we denote by $\Aut(D)$ the group of
biholomorphisms of $D$, by $\Aff(D)$ the  group of
affine biholomorphisms of $D$, and by $GL(D)$ the  group of linear biholomorphisms of $D$.

It is known that the group $\Aut(B_{n+1})$ is connected, and that it
consists of fractional linear transformations (cf.,
e.g.,~\cite[Section 2.1]{Knapp} or~\cite[Chapter 2]{Rudin}). In particular, the stabilizer of $0$
in $\Aut(B_{n+1})$ is the unitary group $U(n+1)$, and is a maximal
compact subgroup of $\Aut(B_{n+1})$. 
 When $n=0$,  $\Aut(B_1)$ is given in~\eqref{def:AutD}. The group
of biholomorphisms $\Aut(\Uc_1)$ of $\Uc_1$ consists of fractional
linear transformations, and may be described as  
\[
\Aut(\Uc_1) =\Set{z \mapsto \frac{a z+b}{c z+d}\colon a,b,c,d\in \R, ad-bc=1}.
\]

Let us now describe some notable subgroups of $\Aut(\Uc_{n+1})$. 
First of all, when $n>0$, consider the Heisenberg
group  $\Hb_n\coloneqq \C^n\times \R$, endowed with the product
\[
(\zeta,x)(\zeta',x')\coloneqq (\zeta+\zeta',x+x'+ 2 \Im \langle \zeta\vert \zeta'\rangle)
\]
for every $(\zeta,x),(\zeta',x')\in \Hb_n$, which acts freely on
$\C^n\times \C$, $\Uc_{n+1}$, and simply transitively on $\partial \Uc_{n+1}$, by the
affine transformations: 
\[
(\zeta,x)\cdot (\zeta',z')\coloneqq (\zeta+\zeta', x+i\abs{\zeta}^2+ 2 i\langle \zeta'\vert\zeta\rangle)
\]
for every $(\zeta,x)\in \Hb$ and for every $(\zeta',z')\in \C^n\times
\C$, where $\langle\,\cdot\,\vert \,\cdot\,\rangle$ denotes the (hermitian)
scalar product on $\C^n$.   With a slight abuse of notation, we
denote by $\Hb_0$ the group $\R$ of horizontal translations on
$\C_+=\Uc_1$.  Then, for all $n\ge0$ we identify $\Hb_n$ with a
subgroup of $\Aut(\Uc_{n+1})$ by means of its action on
$\Uc_{n+1}$. 

\begin{lem}\label{lem:2.1}
  The following properties hold:
  \begin{itemize}
  \item[(i)]  (the group  of linear automorphisms  of $\Uc_{n+1}$)
    \[
    	GL(\Uc_{n+1})=\Set{(\zeta,z)\mapsto (R U\zeta, R^2 z)\colon R>0,
      U\in U(n) };
    \]
  \item[(ii)] (the group  of affine automorphisms  of $\Uc_{n+1}$)
    \[
    \Aff(\Uc_{n+1})= \Hb_n \rtimes GL(\Uc_{n+1});
    \]
      \item[(iii)]  the group 
$G_T\coloneqq \Hb_n \rtimes \Set{(\zeta,z)\mapsto (R \zeta, R^2 z)\colon R>0}$
is solvable and acts simply transitively on $\Uc_{n+1}$.
\end{itemize}   
\end{lem}

\begin{proof}
	 Assertion (ii) follows from~\cite[Proposition 2.1]{Murakami}. It is then clear that $G_T$ is the semi-direct product of $\Hb_n$ and the (abelian) group $G_\delta\coloneqq\Set{(\zeta,z)\mapsto (R \zeta, R^2 z)\colon R>0}$ of dilations of $\Uc_{n+1}$, so that it is solvable. Since $\Hb_n$ acts simply transitively on the translates of $\partial \Uc_{n+1}$ and $G_\delta$ acts simply transitively on $\Set{(0,i h)\colon h>0}$, (iii) follows.
	
	Finally, observe that~\cite[Proposition 2.2]{Murakami} shows
        that   $GL(\Uc_{n+1})$   consists of the linear automorphisms of $\C^n\times \C$ of the form $A\times B_\C$ with $A\in GL(\C^n)$ and $B\in GL(\R)$ such that $B\R_+\subseteq \R_+$ and $\norm{A\zeta}^2=B\norm{\zeta}^2$ for every $\zeta\in \C^n$. Then, there is $R>0$ such that $Bx=R^2 x$ for every $x\in \R$, and $A=R U$ for some $U\in U(n)$, whence (ii). 
\end{proof}

\begin{prop}\label{prop:27}
	The group $\Aut(\Uc_{n+1})$ is generated by $\Aff(\Uc_{n+1})$ and the inversion
	\[
	\iota\colon \Uc_{n+1}\ni (\zeta,z)\mapsto \left(-\frac{i
            \zeta}{z}, -\frac 1  z  \right) \in \Uc_{n+1}.
	\]
	In addition, $\Aff(\Uc_{n+1})$ is amenable.
\end{prop}

Recall that a group $H$ is amenable if it admits a right-invariant
mean, that is, a linear functional $\mathtt{m}\colon \ell^\infty(H)\to
\C$ such that $\mathtt{m}(\chi_H)=1$ and
$\mathtt{m}(f(\,\cdot\,h))=\mathtt{m}(f)$ for every $f\in
\ell^\infty(H)$ and for every $h\in H$. See~\cite{Pier} for several
characterizations of amenability and a thorough treatment of this
subject. 

For future reference, let us also note that the  complex Jacobian $J\iota$ of $\iota$ is
\begin{equation}\label{eq:4}
	(J\iota)(\zeta,z)= \frac{1}{i^n z^{n+2}} 
\end{equation}
for every $(\zeta,z)\in D$.

\begin{proof}
	The first assertion follows, for example, from~\cite[Lemma
        2.1]{Arcozzietal}, or~\cite[Theorem 6.1]{Dorfmeister}.
	Concerning the second assertion, observe that (i) and (ii) of
  Lemma~\ref{lem:2.1}  show  that   $\Aff(\Uc_{n+1})$ is the
  semi-direct product of the  nilpotent Lie group $\Hb_n$ and the
  group $GL(\Uc_{n+1})$, which in turn is the direct product of an
  abelian group (the group of the dilations) and the compact group
  $U(n)$. The assertion then follows from~\cite[Propositions 12.1 and
  13.4, and Corollary 13.5]{Pier}. 
\end{proof}

\begin{deff}   
We denote by $\widetilde \Aut(\Uc_{n+1})$, or simply $\widetilde \Aut$ if
there is no fear of confusion, the universal covering of
$\Aut(\Uc_{n+1})$.  We endow $\widetilde \Aut(\Uc_{n+1})$ with the action
on $\Uc_{n+1}$ induced by the natural action of $\Aut(\Uc_{n+1})$ on
$\Uc_{n+1}$. For every $s\in \R$, consider the continuous
representation $\widetilde U_s$ of $\widetilde \Aut(\Uc_{n+1})$ in
$\Hol(\Uc_{n+1})$  defined by 
\begin{equation}\label{Us:def}
\widetilde U_s(\phi) f(\zeta,z)\coloneqq f(\phi^{-1}(\zeta,z))
(J\phi^{-1})^{s/(n+2)} (\zeta,z),
\end{equation}
where  $(\phi,(\zeta,z))\mapsto (J\phi^{-1})^{s/(n+2)}(\zeta,z)$ may
be (unambiguously) defined as a continuous function on the simply connected 
manifold $\widetilde \Aut(\Uc_{n+1})\times \Uc_{n+1}$, in such a way that
$(J I^{-1})^{s/(n+2)}(0, i)=1$, where $I$ denotes the identity in
$\widetilde \Aut(\Uc_{n+1})$.\footnote{The exponent $s/(n+2)$ may
  seem peculiar at first, but may be justified observing that
  $\widetilde U_s(\delta_R) f =(f\circ \delta_R^{-1})R^{-s}$ for every
  $R>0$ and for every $s\in \R$, identifying the (simply connected)
  group of the dilations $\delta_R\colon (\zeta,z)\mapsto (R\zeta, R^2
  z)$   with a subgroup of $\widetilde \Aut(\Uc_{n+1})$. Besides that,
  this definition is quite common in the literature since $n+2$ may be
  interpreted as the `genus' of $\Uc_{n+1}$ as a symmetric domain.}
 In particular, $(J\phi^{-1})^{s/(n+2)}$ is holomorphic on
$\Uc_{n+1}$ for every $\phi \in \widetilde \Aut(\Uc_{n+1})$ and for
every $s\in \R$. 
\end{deff}

Notice that the subgroup 
$G_T$ defined in  Lemma~\ref{lem:2.1} (iii) is simply connected, so
that it may be identified with a subgroup of $\widetilde \Aut(\Uc_{n+1})$. Then, $\widetilde U_s$ induces a representation $U_s$
of $G_T$. Observe that 
\[
U_s(\phi) f (\zeta,z)= f (\phi^{-1}(\zeta,z)) \abs{(J\phi^{-1})^{s/(n+2)}(0,i)}
\]
for every $\phi\in G_T$, for every $f\in \Hol(\Uc_{n+1})$, and for
every $(\zeta,z)\in \Uc_{n+1}$, and that the same expression may be
employed to define a continuous representation $U_s$ of $\Aff(\Uc_{n+1})$ in
$\Hol(\Uc_{n+1})$.\footnote{Notice that, as long as $\phi\in G_T$, $J\phi^{-1}$ is positive on $\Uc_{n+1}$, so that the absolute value is redundant. If we extend $U_s$ to $\Aff(\Uc_{n+1})$, however, then the absolute value is necessary to make sure that $U_s$ is well defined  when $\Aff(\Uc_{n+1})\neq G_T$, that is, when $n>0$.}

\subsection{Bergman Spaces}\label{sec:2:3}

\begin{deff}
	For every $p\in [1,\infty]$ and for every $s\in \R$, we define
	\[
	L^p_s(\Uc_{n+1})\coloneqq \Set{f \text{ measurable}\colon \int_{\Uc_{n+1}} \abs{f(\zeta,z)}^p (\Im z-\abs{\zeta}^2)^{p s-1}\,\dd (\zeta,z)<\infty  },
	\]
	with the obvious modification when $p=\infty$,
         that is,
\[
	L^\infty_s(\Uc_{n+1})\coloneqq \Set{f \text{ measurable}\colon
 (\zeta,z)  \mapsto (\Im z-\abs{\zeta}^2)^s {f(\zeta,z)} \in L^\infty   },       
\]
       and $A^p_s(\Uc_{n+1})\coloneqq L^p_s(\Uc_{n+1})\cap \Hol(\Uc_{n+1})$.
\end{deff}

Then, $A^p_s(\Uc_{n+1})$ is a Banach space, and embeds continuously into $\Hol(\Uc_{n+1})$. It is a   Hilbert space when $p=2$. In addition, $A^p_s(\Uc_{n+1})\neq \Set{0}$ if and only if $s>0$ or $p=\infty$ and $s\Meg 0$.

The reproducing kernel of $A^2_s(\Uc_{n+1})$, $s>0$, (the `weighted Bergman kernel') is
\begin{equation}\label{eq:9}
K_s\colon \Uc_{n+1}\times \Uc_{n+1}\ni ((\zeta,z),(\zeta',z'))\mapsto c_s \left(\frac{z-\overline{z'}}{2 i}-\langle \zeta\vert\zeta'\rangle  \right)^{-(n+1+2s)},
\end{equation}
where $c_s=\frac{(2s+n)\cdots(2s)}{4 \pi^{n+1}}$. Notice that the power is well
defined since $\Re \left(\frac{z-\overline{z'}}{2 i}-\langle
  \zeta\vert\zeta'\rangle \right)>0 $ for every
$(\zeta,z),(\zeta',z')\in \Uc_{n+1}$.   For simplicity of notation,
we also write
\begin{equation}\label{rho:def}
\rho(\zeta,z) = \Im z -\abs{\zeta}^2
\end{equation}
so that $K_s((\zeta,z),(\zeta,z))= c_s \rho(\zeta,z)^{-(n+1+2s)} $.  

We denote by $\Pc_s$ the corresponding projector (the `weighted Bergman projector'), so that
\[
\Pc_s f(\zeta,z)= \int_{\Uc_{n+1}} f(\zeta',z') K_s((\zeta,z),(\zeta',z')) (\Im z'-\abs{\zeta'}^2)^{2 s-1}\,\dd(\zeta',z')
\]
for every $f\in C_c(\Uc_{n+1})$ (say).

We shall now describe the spaces $\Pc_{s'}(L^{p}_s(\Uc_{n+1}))$ for suitable values of $s,s'$. In order to do that, we shall need to construct some Besov spaces of analytic type on $\Hb_n$ and an associated extension operator.

\begin{deff}
	For every compact subset $K$ of $\R_+^*$, we define $\Sc_{\R_+^*}(\Hb_n,K)$ as the image of   the mapping
	\[
	\Gc\colon C^\infty_c(K)\ni \phi\mapsto \bigg[(\zeta,x)\mapsto \int_{0}^{+\infty} \phi(\lambda) \ee^{ \lambda(i x-\abs{\zeta}^2)}\,\dd \lambda \bigg]\in \Sc(\Hb_n),
	\]
	with the induced topology. We then define $\Sc_{\R_+^*}(\Hb_n)$ and $\Sc_{\R_+^*,L}(\Hb_n)$ as the inductive limits of the spaces $\Sc_{\R_+^*}(\Hb_n,K)$ and $\Sc(\Hb_n)*\Sc_{\R_+^*}(\Hb_n,K)$, respectively, as $K$ runs through the set of compact subsets of $\R_+^*$.
	
	We define $\Sc'_{\R_+^*,L}(\Hb_n)$ as the dual of the conjugate of $\Sc_{\R_+^*,L}(\Hb_n)$.
\end{deff}

Observe that $\Sc_{\R_+^*}(\R)=\Sc_{\R_+^*,L}(\R)=\Fc^{-1}(C^\infty_c(\R_+^*))$ when $n=0$. In this case, $\Gc$ is essentially the inverse Fourier transform. Also in the general case one may interpret $\Gc$ (essentially) as the inverse Fourier transform, cf.~\cite[Chapter 4]{CalziPeloso} for more details.

In order to better explain the meaning of all these constructions, let us observe that $\Hb_n$, identified with the boundary of $\Uc_{n+1}$, inherits the structure of a CR manifold. In other words, the complex tangent spaces (that is, the largest complex vector subspaces of the real tanget spaces) of $\partial \Uc_{n+1}$ have all the same complex dimension (namely, $n$). This is related to the fact that the left translations of $\Hb_n$ act on $\partial \Uc_{n+1}$ by  affine biholomorphisms. 
This CR structure then allows to define CR functions on $\Hb_n$, that
is, functions which satisfy  the  tangential   Cauchy--Riemann
equations. Then, the space $\Sc_{\R_+^*,L}(\Hb_n)$ may be equivalently
characterized as the space of CR Schwartz functions on $\Hb_n$ whose
`spectrum' is contained in $\R_+^*$, where the spectrum of a CR $f\in
\Sc(\Hb_n)$ may be defined as $\overline{\bigcup_{\zeta\in \C^n}
  \Supp{\Fc(f(\zeta,\,\cdot\,))}}$, where $\Fc$ denotes the
(Euclidean) Fourier transform in the second variable.   It turns out that the elements of $\Sc_{\R_+^*,L}(\Hb_n)$ may also be characterized as the restrictions to $\Hb_n\cong \partial \Uc_{n+1}$ of a class of entire functions on $\C^{n+1}$ which satisfy suitable growth and decay conditions (in a Paley--Wiener--Schwartz fashion). The elements of $\Sc_{\R_+^*}(\Hb_n)$ then correspond to the entire functions which depend only on the last variable. Cf.~\cite{PWS} for a more detailed discussion of this topic.

We denote by $\langle\,\cdot\,\vert \,\cdot\,\rangle$ the sesquilinear pairing  between $\Sc'_{\R_+^*,L}(\Hb_n)$ and $\Sc_{\R_+^*,L}(\Hb_n)$, as well as the conjugate of the pairing between $\Sc_{\R_+^*,L}(\Hb_n)$ and $\Sc'_{\R_+^*,L}(\Hb_n)$.

\begin{deff}
	Take $p\in [1,\infty]$ and $s\in \R$. We define $B^{s}_p(\Hb_n,\R_+^*)$ as the space of $u\in \Sc_{\R_+^*,L}(\Hb_n)$ such that
	\[
	\sum_{j\in \Z} 2^{ p s j} \norm{u* \Gc(\phi_j)}_{L^p(\Hb_n)}^p<\infty,
	\]
	endowed with with the corresponding norm (with the obvious modification when $p=\infty$), where $(\phi_j)$ is a family of elements of $C^\infty_c(\R_+^*)$  such that the $\phi_j(2^{-j}\,\cdot\,) $ stay in a bounded subset of $C^\infty_c(\R_+^*)$ and $\sum_{j\in \Z} \phi_j\Meg 1$ on $\R_+^*$.
\end{deff}

We remark   that this definition does not depend on the choice of $(\phi_j)$ (cf.~\cite[Lemma 4.14]{CalziPeloso}). In addition, $B^{s}_p(\Hb_n,\R_+^*)$ may be interpreted as a suitable closed subspace of a natural Besov space on $\Hb_n$, cf.~\cite[Sections 7 and 8]{Besov} for more details.
Notice that, if $p<\infty$, then $\Sc_{\R_+^*,L}(\Hb_n)$ is dense in $B^s_p(\Hb_n,\R_+^*)$, and the sesquilinear pairing between  $\Sc_{\R_+^*,L}(\Hb_n)$ and $\Sc'_{\R_+^*,L}(\Hb_n)$ induces a canonical sesquilinear pairing between $B^s_p(\Hb_n,\R_+^*)$ and $B^{-s}_{p'}(\Hb_n,\R_+^*)$ which identifies $B^{-s}_{p'}(\Hb_n,\R_+^*)$ with the dual of $B^s_p(\Hb_n,\R_+^*)$ (cf.~\cite[Theorem 4.23]{CalziPeloso}). We may then consider a canonical sesquilinear form between $B^s_p(\Hb_n,\R_+^*)$ and $B^{-s}_{p'}(\Hb_n,\R_+^*)$ for every $p\in [1,\infty]$ and for every $s\in \R$.

\begin{deff}
	Take $p\in [1,\infty]$  and $s>-\frac{n+1}{p}$. Then, $S_{(\zeta,z)}\colon (\zeta',x')\mapsto \frac{n!}{4\pi^{n+1}}\left(\frac{x'+i \abs{\zeta'}^2-\overline z}{2 i}-\langle \zeta'\vert \zeta\rangle  \right)^{-n-1}$ belongs to $B^{s}_{p'}(\Hb_n,\R_+^*)$ for every $(\zeta,z)\in \Uc_{n+1}$ and the linear mapping
	\[
	\Ec\colon B^{-s}_p(\Hb_n,\R_+^*)\ni u \mapsto [(\zeta,z)\mapsto \langle u\vert S_{(\zeta,z)}\rangle]\in A^\infty_{s+(n+1)/p}(\Uc_{n+1})
	\]
	is well defined and continuous. We define $\widetilde A^p_s(\Uc_{n+1})$ as its image, endowed with the corresponding topology.
\end{deff}

Notice that $S_{(\zeta,z)}$ is the boundary value of the Cauchy--Szeg\H o kernel on $\Uc_{n+1}$, that is, the reproducing kernel of the Hardy space $H^2(\Uc_{n+1})$. Thus, integrating the boundary values of an element $f$ of $H^2(\Uc_{n+1})$ against $S_{(\zeta,z)}$ gives $f(\zeta,z)$, so that  the extension operator $\Ec$ is naturally defined.

We endow $\Uc_{n+1}$ with an $\Aut(\Uc_{n+1})$-invariant Riemannian
metric (e.g., the Bergman metric) and the associated distance. For
every $\delta>0$, we say that a family $(\zeta_j,z_j)$ of elements of
$\Uc_{n+1}$ is a $\delta$-lattice if it is maximal for the property
that $d((\zeta_j,z_j),(\zeta_{j'},z_{j'}))\Meg 2\delta$ for every
$j\neq j'$. We may then state the following result.

\begin{prop}\label{prop:23}
	Take  $p\in [1,\infty]$ and $s>-\frac{n+1}{p}$. Then,  the following hold:
	\begin{enumerate}
		\item[\textnormal{(i)}] $P_{s'}$ induces a continuous linear mapping of $L^p_s(\Uc_{n+1})$ onto $\widetilde A^p_s(\Uc_{n+1})$ for every $s'> \frac{s_+}{2}$;
		
		\item[\textnormal{(ii)}]  $A^p_s(\Uc_{n+1})=\widetilde A^p_s(\Uc_{n+1})$ for every $s>0$;
		
		\item[\textnormal{(iii)}] for every $s'> \frac s 2 $ and for every $\delta>0$, and for every $\delta$-lattice $(\zeta_j,z_j) $    on $\Uc_{n+1}$, the mapping
		\[
		\ell^p(J)\ni \lambda \mapsto \sum_j \lambda_j K_{s'}(\,\cdot\,,(\zeta_j,z_j)) \rho(\zeta_j,z_j)^{(n+1)/p'+2 s'-s}\in \widetilde A^{p}_s(\Uc_{n+1})
		\]
		is well-defined and continuous, with locally uniform convergence of the sum when $p=\infty$, and onto when $\delta$ is sufficiently small;
		
		\item[\textnormal{(iv)}] the mapping $\widetilde A^p_s(\Uc_{n+1})\ni f \mapsto \partial_2^k f\in \widetilde A^p_{s+k}(\Uc_{n+1}) $ is an isomorphism for every $k\in \N$;
		
	 \item[\textnormal{(v)}] if $p<\infty$, then 
		\[
		\begin{split}
		\widetilde A^p_s(\Uc_{n+1})&=\Set{f\in A^\infty_{s+(n+1)/p}(\Uc_{n+1})\colon \partial^k_2 f\in A^p_{s+k}(\Uc_{n+1})}\\
			&=\Set{f\in A^\infty_{s+(n+1)/p,0}(\Uc_{n+1})\colon \partial^k_2 f\in A^p_{s+k}(\Uc_{n+1})}
		\end{split}
		\]
		for every $k\in\N$ such that $s+k>0$, where $A^\infty_{s',0}(\Uc_{n+1})\coloneqq \Set{f\in \Hol(\Uc_{n+1})\colon \rho^{s'} f \in C_0(D)  }$, endowed with the norm of $A^\infty_{s'}(\Uc_{n+1})$, for every $s'>0$.
	\end{enumerate}
\end{prop}

Here,  and throughout the paper, we set 
\[
\partial_2 f(\zeta,z)=\frac{\dd}{\dd z} f(\zeta,z)
\]
for every $f\in \Hol(\Uc_{n+1})$ and for every $(\zeta,z)\in \Uc_{n+1}$.
Notice that, when $n=0$, this leads to an abuse of notation, as it is more natural to consider $\C^n\times \C$ as $\C$ (so that there is no `second' variable), than $\Set{0}\times \C$.

\begin{proof}
	(i) This follows from~\cite[Theorem 4.5]{Paralipomena} and~\cite[Corollary 5.11]{CalziPeloso}.
	
	(ii) This follows from~\cite[Corollary 5.11]{CalziPeloso}.
		
	 (iii) This follows from~\cite[Theorem 4.5]{Paralipomena} and~\cite[Corollary 5.11]{CalziPeloso}.

	(iv) This follows from~\cite[Proposition 5.13]{CalziPeloso}.
	
	 (v)  Observe that $\widetilde A^p_s(\Uc_{n+1})\cap A^2_{1/2}(\Uc_{n+1})$ is dense in $\widetilde A^p_s(\Uc_{n+1})$ and contained in $A^\infty_{s+(n+1)/p,0}(\Uc_{n+1})$ (cf.~\cite[Theorem 4.23 and Propositions 5.4 and 3.9]{CalziPeloso}), so that $\widetilde A^p_s(\Uc_{n+1})$ embeds continuously into $A^\infty_{s+(n+1)/p,0}(\Uc_{n+1})$. Thus, the assertion follows by means of (ii) and (iv), applied to $\widetilde A^p_s(\Uc_{n+1})$ and $\widetilde A^\infty_{s+(n+1)/p}(\Uc_{n+1})$.
\end{proof}

For future reference we state the following elementary lemma.

\begin{lem}\label{lem:20}
	Take $p\in [1,\infty]$ and $s\Meg s'>-\frac{n+1}{p}$. Then, for every $h>0$, the mapping
	\[
	\widetilde A^p_s(\Uc_{n+1})\ni f \mapsto f(\,\cdot\,+(0,ih))\in \widetilde A^{p}_{s'}(\Uc_{n+1})
	\]
	is well defined and continuous.
\end{lem}

\begin{proof}
	Observe first that the case $s=s'$ follows from~\cite[Theorem 5.2]{CalziPeloso}, so that we may assume that $s>s'$, that is, $s-s'>0 $. Then, observe that $S_{(0,i h)}\in B^{s-s'}_1(\Hb_n,\R_+^*)$ by~\cite[Lemma 5.1]{CalziPeloso}, so that~\cite[Theorem 4.3 and the remarks following its statement]{Besov} show  that the mapping
	\[
	B^{-s}_p(\Hb_n,\R_+^*)\ni u\mapsto u* S_{(0,i h)}\in B^{-s'}_p(\Hb_n,\R_+^*)
	\] 
	is well defined and continuous. The assertion follows since
	\[
	(\Ec u)(\,\cdot\,+(0, i h))=\Ec (u* S_{(0,i h)}),
	\]
	as one readily verifes (cf., e.g., the proof of~\cite[Theorem 5.2]{CalziPeloso}).
\end{proof}

 We now begin our analysis of  invariant  function
spaces.  We recall that $\widetilde U_s$ is defined in~\eqref{Us:def}
and $\rho(\zeta,z)$ in~\eqref{rho:def}.  The following lemma is well
known, but we state it explicitly for the reader's convenience.
\begin{lem}\label{invariance:lem}
 For every $s>n+1$ and every $\phi \in \Aut(\Uc_{n+1})$,
 \begin{equation}\label{eq:1}
 (\widetilde U_{n+1+2s}(\phi)\otimes \overline{\widetilde
  U_{n+1+2s}(\phi)}) K_s=K_s.
\end{equation}
\end{lem}

\begin{proof}
It suffices to observe that if $s=1/2$ this is just  the standard transformation
rule of the unweighted Bergman kernel.  The general case follows at once.
\end{proof}

\begin{prop}\label{prop:24}
	Take $p\in [1,\infty]$ and $s\in \R$. Then, the following hold:
	\begin{itemize}
		\item[\textnormal{(i)}] $A^p_s(\Uc_{n+1})$ is $\widetilde \Aut$-$\widetilde U_{2(s+(n+1)/p)}$-invariant with its norm;
		
		\item[\textnormal{(ii)}] if $s>-\frac{n+1}{p}$, then $\widetilde A^p_s(\Uc_{n+1})$ is $\widetilde \Aut$-$\widetilde U_{2(s+(n+1)/p)}$-invariant and the  $\widetilde U_{2(s+(n+1)/p)}(\phi)$, as $\phi$ runs through $\widetilde \Aut$, are uniformly bounded on $\widetilde A^p_s(\Uc_{n+1})$.
	\end{itemize}
\end{prop}

\begin{proof}
  (i) Observe first that, applying Lemma~\ref{invariance:lem}  with
  $s=1/2$ and evaluating both sides of~\eqref{eq:1} at $((\zeta,z),(\zeta,z))$, we see that
  \[
        \rho(\phi(\zeta,z))^{-(n+2)} \abs{J \phi(\zeta,z)}^2=\rho(\zeta,z)^{-(n+2)}
	\]
	for every $(\zeta,z)\in \Uc_{n+1}$. 
     Then, if $p\in [1,\infty)$ and $f\in A^p_s(\Uc_{n+1})$,
\begin{align*}
&\int_{\Uc_{n+1}} \abs{\widetilde U_{2(s+(n+1)/p)} (\phi) f(\zeta,z)}^p \rho(\zeta,z)^{ps-1}\,
  \dd(\zeta,z) \\
  & \qquad= \int_{\Uc_{n+1}} \abs{f(\phi^{-1}(\zeta,z))}^p
    \abs{(J\phi^{-1})(\zeta,z)}^{2 + 2(ps-1)/(n+2)} \rho(\zeta,z)^{ps-1}\, \dd(\zeta,z) \\
& \qquad= \int_{\Uc_{n+1}} \abs{f(\zeta,z)}^p \rho(\zeta,z)^{ps-1}\,
  \dd(\zeta,z) .
\end{align*}
The case $p=\infty$ is similar.

(ii) Observe that, by (ii)P of Proposition~\ref{prop:23}, 
$\widetilde A^\infty_{s_\infty} (\Uc_{n+1})= A^\infty_{s_\infty}(\Uc_{n+1})$, where
$s_\infty=s+(n+1)/p>0$, so that and part (i) shows that $\widetilde
A^\infty_{s_\infty}(\Uc_{n+1})$ is $\widetilde \Aut$-$\widetilde
U_{2s_\infty}$-invariant. Now, \cite[Proposition 6.2]{Besov} shows
that
\[
\widetilde A^p_s \cong(\widetilde A^1_{s_0}, \widetilde A^\infty_{s_\infty})_{[1/p]}
\]
(complex interpolation), 
where  $s_0=s-(n+1)/p'>-(n+1)$.  Thus, it
suffices to prove that $\widetilde A^1_s(\Uc_{n+1})$ is $\widetilde
G$-$\widetilde U_{2(s+n+1)}$-invariant whenever $s>-(n+1)$. 

To this aim, observe that, if $(\zeta_j,z_j)$ is a $\delta$-lattice on
$\Uc_{n+1}$ with $\delta$ sufficiently small, then (iii) of
Proposition~\ref{prop:23} shows that the mapping 
	\[
	\Psi_{(\zeta_j,z_j)}\colon \ell^1(J)\ni \lambda \mapsto \sum_j \lambda_j K_{s'}(\,\cdot\,,(\zeta_j,z_j)) \rho(\zeta_j,z_j)^{2 s'-s}\in \widetilde A^1_s(\Uc_{n+1})
	\]
	is continuous and onto, where $s'\coloneqq s+(n+1)/2 >s/2$.
        Then, using the relations
	\[
	\widetilde U_{2(s+n+1)}(\phi) K_{s'}(\,\cdot\,,(\zeta_j,z_j))= K_{s'}(\,\cdot\,, \phi(\zeta_j,z_j)) \overline{J\phi(\zeta_j,z_j)}^{2( s+n+1)/(n+2)}
	\]
	and
	\[
	\rho(\zeta_j,z_j)^{2 s'-s}=\rho(\phi(\zeta_j,z_j))^{2 s'-s} \abs{J\phi(\zeta_j,z_j)}^{2(s-2s')/(n+2)}
	\]
	for every $\phi\in \Aut(\Uc_{n+1})$ and for every $j\in J$, which follow
        from   Lemma~\ref{invariance:lem},   we see that there is a unimodular function
        $u_\phi\colon J\to \C$ such that 
	\[
	\widetilde U_{2(s+n+1)}(\phi) \Psi_{(\zeta_j,z_j)}(\lambda)= \Psi_{(\phi(\zeta_j,z_j))}(u_\phi \lambda)
	\]
	for every $\phi \in \Aut(\Uc_{n+1})$ and for every $\lambda$ with finite
support, hence for every $\lambda\in \ell^1(J)$. The assertion follows
from the fact that there is $C>1$ such that, fixing a norm on
$\widetilde A^1_s(\Uc_{n+1})$, 
	\[
	\frac{1}{C} \norm{\lambda }_{\ell^1(J)/\ker\Psi_{(\phi(\zeta_j,z_j))}}
        \meg \norm{\Psi_{(\phi(\zeta_j,z_j))} (\lambda) }_{\widetilde A^1_s(\Uc_{n+1})}
        \meg C  \norm{\lambda }_{\ell^1(J)/\ker \Psi_{(\phi(\zeta_j,z_j))}}
	\]
for every $\lambda\in \ell^1(J)$ and for every $\phi\in G$, as
the proof of Proposition~\ref{prop:23} shows (cf.~\cite[Theorem
3.23]{CalziPeloso}). 
\end{proof} 

\subsection{Decent and Strongly Decent Spaces}

\begin{deff}\label{def:6}
	Let $X,Y$ be two locally convex 
        spaces such that $X\subseteq Y$ set-theoretically. Then, we
        say that $X$ is $Y$-decent if there is a continuous linear
        functional on $Y$ which induces a non-zero continuous linear
        functional on $X$. 
	
	We say that $X$ is strongly $Y$-decent if the set of
        continuous linear functionals on $X$ which extend to
        continuous linear functionals on $Y$ is dense in the weak dual
        topology of $X'$.
	
	We say that $X$ is $Y$-saturated if it contains the intersection of the kernels in $Y$ of the continuous linear functionals on $Y$ which induce continuous linear functionals on $X$.
\end{deff}

In the sequel, $Y$ will always be the space $\Hol(D)$ and $X$ a
  semi-Banach space   of $Y$, so that we shall simply say that $X$
(or its seminorm) is (strongly) decent or saturated  if it is
(strongly) $\Hol(D)$-decent or $\Hol(D)$-saturated, respectively, for
simplicity. We shall nonetheless state the results of this Subsection
in a somewhat more general context. 

Notice that if $X$ is strongly $Y$-decent, then it is $Y$-decent if and only if it is non-trivial (as a topological vector space, that is, it has a non-trivial topology).

\begin{prop}\label{prop:31}
	Let $X$ be a complete seminormed space and $Y$ a Fréchet space such that $X\subseteq Y$. Let $G$ be a group of automorphisms of $Y$ which induce automorphisms of $X$. Then, the following hold:
	\begin{enumerate}
		\item[\textnormal{(i)}] $X$ is $Y$-decent if and only if there is a closed $G$-invariant vector subspace $V$ of $Y$ such that the canonical mapping $X\to Y/V$ is continuous and non-trivial;
		
		\item[\textnormal{(ii)}] $X$ is strongly $Y$-decent if and only if there is a closed $G$-invariant vector subspace $V$ of $Y$ such that $X\cap V$ is the closure of $\Set{0}$ in $X$ and the canonical mapping $X\to Y/V$ is continuous;
		
		\item[\textnormal{(iii)}] $X$ is strongly $Y$-decent and $Y$-saturated if and only if the ($G$-invariant) closure $V$ of $\Set{0}$ in $X$ is closed  in $Y$, and the canonical mapping $X\to Y/V$ is continuous.
	\end{enumerate} 
\end{prop}

Notice that, if $X$ is strongly $Y$-decent and $V$ is as in (ii), then
$X+V$, endowed with the seminorm which is $0$ on $V$ and induces the
given seminorm on $X$, is strongly $Y$-decent, $Y$-saturated, and
$G$-invariant. In other words, every strongly $Y$-decent space has a
`saturation'. 

Notice that, taking $G=\Set{I}$, we get a characterization of (strongly) $Y$-decent spaces.

\begin{proof}
	Assume that $X$ is (strongly) $Y$-decent, and let $W$ be the space of continuous linear functionals on $Y$ which induce continuous linear functionals on $X$. Observe that $W$ is $G$-invariant since $G$ induces automorphisms of both $X$ and $Y$. In addition, if $V=\bigcap_{L\in W} \ker W$, then $V$ is a closed $G$-invariant subspace and $X\not\subseteq V$ (resp.\ $X\cap V$ is the closure of $\Set{0}$ in $X$, since the canonical image of $W$ in $X'$ is dense in the weak dual topology). Let us prove that the mapping $X\to Y/V$ is continuous. Notice that it will suffice to prove that the canonical mapping $X/(X\cap V)\to Y/V$ is continuous. 
	Since both $X/(X\cap V)$ and $Y/V$ are Fréchet spaces, we may use the closed graph theorem. Then, let $(x_j)$ be a sequence in $X$ and take $x\in X$ and $y\in Y$ so that $x_j+X\cap V$ converges to $x+X\cap V$ in $X/(X\cap V)$, while $x_j+V$ converges to $y+V$ in $Y/V$. 
	Then, $\langle L,x_j\rangle$ converges both to $\langle L, x\rangle$ and to $\langle L,y\rangle $ for every $L\in W$, so that $x-y\in V$. Then, the canonical image of $x+X\cap V$ in $Y/V$ is $y+V$, whence our claim.
	
	Conversely, assume that there is a  closed vector subspace $V$ of $Y $ such that the canonical mapping $X\to Y/V$ is continuous and non-trivial. Then, there is $L\in (Y/V)'$ which does not vanish on $\pi(X)$, where $\pi\colon Y\to Y/V$ is the canonical mapping. Then, $L\circ \pi$ induces a non-zero continuous linear functional on both $X$ and $Y$, so that $X$ is $Y$-decent. This completes the proof of (i).
	
	Then, assume that there is a closed vector subspace $V$ of $Y$ such that $X\cap V$ is the closure of $\Set{0}$ in $X$ and the canonical mapping $X\to Y/V$ is continuous. Then, each element of the polar $W$ of $V$ in $Y'$ induces a continuous linear functional on $X$, and the canonical image of $W$ in $X'$ is dense in the weak dual topology (as $X\cap V$ is the closure of $\Set{0}$ in $X$). Thus, $X$ is strongly $Y$-decent.
	
	The proof of (iii) follows from that of (ii) and the definition of $Y$-saturated spaces.
\end{proof}

\begin{prop}\label{prop:32}
 Let $X$ be a   semi-Banach space and $Y$  a separable Fréchet space such that $X\subseteq Y$. Assume that $X$ is strongly $Y$-decent and $Y$-saturated, and denote by $V$ the closure of $\Set{0}$ in $X$. Assume that the Hausdorff space associated with $X$ (namely, $X/V$) is reflexive. Then, $X$ is separable.
	
	In addition, let $G$ be a locally compact group and $\pi$ be a continuous representation of $G$ in $Y$ such that $\pi(g)$ induces a continuous automorphism of $X$ for every $g\in G$.
	Then,   $\pi$ induces a continuous representation of $G$ in $X$ (or, equivalently, in $X/V$).
\end{prop}

 Notice that the assumptions are satisfied if $Y=\Hol(D)$ and $X$ is a semi-Hilbert space. 

\begin{proof}
	 By Proposition~\ref{prop:31},  there is a closed subspace $V$ of $Y$ such that $X\cap V$ is the closure of $\Set{0}$ in $X$ and the canonical mapping $L\colon X\to Y/V$ is continuous.  By~\cite[Corollary 1 to Proposition 6 of Chapter III, \S\ 3, No.\ 4]{BourbakiTVS}, the weak dual topology of $(Y/V)'$ is separable. Since $\trasp L$ is continuous with dense image for the weak dual topologies, this proves that $X'$ is separable in the weak dual topology. Since $X$ is reflexive, this proves that $X'$ is separable in the strong dual topology, so that also $X$ is separable.
	
	Next, observe that the mapping $G\ni g\mapsto \langle L, \pi(g) x\rangle$ is continuous for every $x\in X$ and for every $L\in V^\circ$. Since the canonical image of $V^\circ$ is dense in the weak topology of $X'$, hence in the strong topology of $X'$, thanks to Proposition~\ref{prop:31},~\cite[Corollary 2 to Proposition 18 of Chapter VIII, \S\ 4, No.\ 6]{BourbakiInt2} shows that $\pi$ induces a continuous representation of $G$ in $X/V$, whence the result.
\end{proof}

\section{Invariant Reproducing Kernel Hilbert Spaces }\label{sec:3}

We define, for every $s\in  \C$,
\[
B^s_{(\zeta',z')}(\zeta,z)\coloneqq \left(\frac{z-\overline {z'}}{2i} - \langle \zeta\vert \zeta'\rangle  \right)^s
\]
for every $(\zeta,z),(\zeta',z')\in \Uc_{n+1}$, so that
$K_s((\zeta,z),(\zeta',z'))=c_s B^{-n-1-2s}_{(\zeta',z')}(\zeta,z)$
for $s>0$, cf.~\eqref{eq:9}. 

 Observe that, by the results in~\cite{VergneRossi}, $B^s$ is a positive kernel, that is, is the reproducing kernel of some RKHS contained in
$\Hol(\Uc_{n+1})$ if and only if $s\Meg 0$.

\begin{deff}
	Take $s\Meg 0$. Then, we define $\Ac_s$ as the RKHS associated with $B^{-s}$. In other words, $\Ac_s$ is the completion of the space of finite linear combinations of the $B^{-s}_{(\zeta,z)}$, $(\zeta,z)\in \Uc_{n+1}$, endowed with the scalar product defined by
	\[
	\big\langle B^{-s}_{(\zeta,z)}\big\vert B^{-s}_{(\zeta',z')}\big\rangle_{\Ac_s} \coloneqq B^{-s}_{(\zeta',z')}(\zeta,z)
	\]
	for every $(\zeta,z),(\zeta',z')\in D$.
\end{deff}

\begin{prop}\label{prop:33}
	If $s>0$, then $\Ac_s=\widetilde A^2_{(s-n-1)/2}(\Uc_{n+1})$ (with equivalent norms). In addition, $\Ac_0$ is the space of constant functions.
\end{prop}

 Notice that~\cite{Arcozzietal} provides yet another description of $\Ac_s$, for $s>0$, which is somewhat analogous to the one which may be obtained by means of (v) of Proposition~\ref{prop:23}. 

\begin{proof}
	Observe that $B^{-s}= c_{(s-n-1)/2}^{-1} K_{(s-n-1)/2} $ for $s>n+1$ (cf.~\eqref{eq:9}), and that $K_{(s-n-1)/2}$ is the reproducing kernel of $A^2_{(s-n-1)/2}(\Uc_{n+1})=\widetilde A^2_{(s-n-1)/2}(\Uc_{n+1})$ (cf.~Proposition~\ref{prop:23}), so that the assertion is clear in this case. 
	Next, take $s>0$, and take $k\in \N$ so large that $s+2 k>n+1$. Observe that\footnote{Recall that $\partial_2 f(\zeta,z)=\frac{\dd}{\dd z} f(\zeta,z)$.} 
	\[
	(\partial_2^k\otimes \overline{\partial_2^k}) B^{-s}= (2i)^{-2 k} (-s)\cdots (-s-2 k) B^{-s-2 k },
	\]
	so that $\partial_2^k$ induces a continuous linear mapping of
        $ \Ac_s$ \emph{onto} $\Ac_{s+2k}$. In addition,
        $B^{-s}_{(\zeta,z)}\in \widetilde A^2_{(s-n-1)/2}(\Uc_{n+1})$
        for every $(\zeta,z)\in \Uc_{n+1}$, thanks to~\cite[Lemma 5.15  and Example 2.12]{CalziPeloso}, and  $\partial_2^k$ induces an isomorphism of $\widetilde A^2_{(s-n-1)/2}(\Uc_{n+1})$ onto $\widetilde A^2_{(s+2k-n-1)/2}(\Uc_{n+1})=\Ac_{s+2k}$ by Proposition~\ref{prop:23}. It then follows that $\Ac_s=\widetilde A^2_{(s-n-1)/2}(\Uc_{n+1})$.
	
	Finally, it is clear that $\Ac_0$ is the space of constant functions.
\end{proof}

\begin{teo}\label{teo:1}
	Take $s\in \R$. If $s \Meg 0$, then $\Ac_{s}$ is $\widetilde \Aut$-$\widetilde U_s$-invariant with its norm.
	
 Conversely, let $H$ be a non-trivial Hilbert space continuously embedded in $\Hol(\Uc_{n+1})$ such that $U_s$ induces a bounded (resp.\ isometric) representation of $G_T $ in $H$.  
	Then, $  s\Meg 0$ and $H=\Ac_{s}$ with equivalent (resp.\ proportional) norms. 
\end{teo}

In comparison with~\cite[Theorem 5.1]{Arazy}, we observe that our invariance condition is considerably weaker, since we require invariance only on $G_T$ and not on $\Aut(\Uc_{n+1})$. The proof is essentially the same, though.
We point out that we replace the `weak integrability' condition
considered in~\cite[Theorem 5.1]{Arazy} with the requirement that $H$
embed continuously into $\Hol(\Uc_{n+1})$. However, as  we shall prove in Proposition~\ref{prop:38}, this `weak integrability' condition is actually  equivalent to the continuity of the embedding of $H$ into $\Hol(\Uc_{n+1})$, thanks to Cauchy's theorem (cf.~also the proof of~\cite[Theorem 5.1]{Arazy}, on which Proposition~\ref{prop:38} is based).

\begin{proof}
	\textsc{Step I.} Let us  prove that, if $s \Meg 0$, then $\Ac_{ s}$ and its norm are $\widetilde U_{s}$-invariant. Notice that, by the definition of $\Ac_s$, this is equivalent to the relation
	\[
	(\widetilde U_{s}(\phi)\otimes \overline{\widetilde U_{s}(\phi)}) B^{-s}=B^{-s}
	\]
	for every   $\phi\in \widetilde \Aut$. This latter fact is clear
        for $s>n+1$, thanks to Lemma~\ref{invariance:lem}, and
        then extends to general $s$ by taking powers.

	\textsc{Step II.} Take $H$  as in the statement, and define
	\[
	C\coloneqq \sup_{\phi\in G_T} \norm{U_{s}(\phi)}_{\Lin(H)},
	\]
	so that $C$ is finite (resp.\ $1$). Take a right-invariant mean $\mathtt{m}$ on $\ell^\infty(G_T)$ (cf.~Proposition~\ref{prop:27}), and define
	\[
	\langle f\vert g\rangle'_H\coloneqq \mathtt{m}(\phi\mapsto \langle U_{ s}(\phi)f\vert U_{s} (\phi) g\rangle_H  )
	\]
	for every $f,g\in H$, so that $\langle\,\cdot\,\vert \,\cdot\,\rangle'_H$ is a well-defined $U_{s}$-invariant scalar product on $H$. In addition,
	\[
	\frac{1}{C}\norm{f}_H\meg\norm{f}'_H\meg C \norm{f}_H
	\]
	for every $f\in H$.
	Let $K$ be the reproducing kernel of $H$, with respect to the scalar product $\langle\,\cdot\,\vert \,\cdot\,\rangle'_H$, and observe that $K$ is $(U_{ s}\otimes \overline{U_{ s}})$-invariant. Hence, the mapping 
	\[
	((\zeta,z),(\zeta',z'))\mapsto K((\zeta,z),(\zeta',z')) B^{\vect s}_{(\zeta',z')}(\zeta,z)
	\]
	is invariant under composition by the elements of $G_T$, thanks to \textsc{step I}. Since $G_T$ acts transitively on $\Uc_{n+1}$, it then follows that there is a constant $C'>0$ such that
	\[
	K((\zeta,z),(\zeta,z))=C' B^{-s}_{(\zeta,z)}(\zeta,z)
	\]
	for every $(\zeta,z)\in \Uc_{n+1}$. Since the function 
	\[
	\Uc_{n+1}\times c(\Uc_{n+1})\ni ((\zeta,z),(\zeta',z'))\mapsto K((\zeta,z),c(\zeta',z')) B^{\vect s}_{c(\zeta',z')}(\zeta,z)\in\C
	\]
	is holomorphic (where $c\colon (\zeta,z)\mapsto \overline{( \zeta, z)}$), we see that
	\[
	K((\zeta,z),(\zeta',z'))=C' B^{-  s}_{(\zeta',z')}(\zeta,z)
	\]
	for every $(\zeta,z),(\zeta',z')\in \Uc_{n+1}$. In particular,
        $B^{-  s}$ is a positive kernel, so that $s\Meg 0$ (cf.~the remarks at the beginning of Section~\ref{sec:3}). It follows that  $H=\Ac_{ s}$ and
	\[
	\norm{f}'_H=C'^{-1/2}\norm{f}_{\Ac_{ s}}
	\]
	for every $f\in H$. 
\end{proof}

\begin{oss}\label{oss:2}
	Take $s,s'>0$. Then, the unitary representations $U_s$ and $U_{s'}$ of $\Aff(\Uc_{n+1})$ into $\Ac_s$ and $\Ac_{s'}$, respectively, are unitary equivalent. An intertwining operator is given by the Riemann--Liouville operator $f\mapsto f* I^{(s-s')/2}$, where $I^{(s-s'/2)}$ is the tempered distribution supported on $\R_+$ whose Laplace transform is $(\,\cdot\,)^{(s'-s)/2}$ on $\R_+^*$.
	
	The unitary representations $\widetilde U_s$ and $\widetilde
        U_{s'}$ of   $\widetilde \Aut(\Uc_{n+1})$   into $\Ac_s$
        and $\Ac_{s'}$, respectively, are   unitarily  equivalent if and only if $s=s'$.
\end{oss}

\begin{proof}
	The first assertion follows from Propositions~\ref{prop:23} and~\ref{prop:33} and from the homogeneity of the distribution $I^{(s-s'/2)}$.
	
	For what concerns the second assertion, observe that, since $U_s$ and $U_{s'}$, restricted to $G_T$, may be considered as irreducible subrepresentations of $\widetilde U_s$ and $\widetilde U_{s'}$, respectively (cf.~Theorem~\ref{teo:1}), if $\widetilde U_s$ and $\widetilde U_{s'}$ are (unitarily) equivalent, then $f\mapsto f* I^{(s-s')/2}$ must be an intertwining operator by Schur's lemma (cf.~\cite[Theorem 2 of \S\ 21, No.\ 3]{Naimark}). In addition, if $\widetilde K$ is the stabilizer of $(0,i)$ in $\widetilde \Aut$, then the $\widetilde K$-$U_s$-orbit of $B^{-s}_{(0, i)}$ in $ \Ac_s$ is one-dimensional by Lemma~\ref{invariance:lem}. On the contrary, $B^{-s}_{(0,i)}*I^{(s-s')/2}=c_{s,s'} B^{(s+s')/2}_{(0,i)}$ for some non-zero $c_{s,s'}$ (cf.~the proof of~\cite[Lemma 5.15]{CalziPeloso}), and the $\widetilde K$-$U_{s'}$-orbit of $B^{(s+s')/2}_{(0,i)}$ is \emph{not} one-dimensional unless $s=s'$ (either use the proof of Theorem~\ref{teo:1} to show that $\widetilde U_{s'}(\iota) B^{(s+s')/2}_{(0,i)}$ is not a scalar multiple of $B^{(s+s')/2}_{(0,i)}$, or use Proposition~\ref{prop:38} transferring the problem to the bounded realization of $D$). 
\end{proof}

We conclude this section with some related remarks that will prove useful in the following sections.

\begin{deff}
	For every $s\in \R$ and for every $k\in \N$ such that $s+2 k\Meg0$, we  define\footnote{Recall that $\partial_2 f(\zeta,z)=\frac{\dd}{\dd z}f(\zeta,z)$.}
	\[
	\Ac_{s,k}\coloneqq \Set{f\in \Hol(\Uc_{n+1})\colon \partial_2^k f\in \Ac_{s+2k}},
	\]
	endowed with the corresponding prehilbertian seminorm.

    We define   $\widehat\Ac_{s,k}$ as the Hausdorff space associated with $\Ac_{s,k}$, that is, $\Ac_{s,k}/\ker\partial_2^k$.
\end{deff}

We remark that, using \cite[Theorem 5.5]{Arcozzietal}
it is easy to check that 
 $\widehat{\Ac}_{0,1}$  is canonically isomorphic to the classical Dirichlet space, for all $n\ge0$.

\begin{cor}\label{cor:1}
	Take $s\in \R$ and $k\in \N$. If $s+2 k\Meg 0$, then $\Ac_{s,k}$ is a semi-Hilbert space and is $\Aff(\Uc_{n+1})$-$U_s$-invariant with its seminorm.
	
	Conversely, let  $H$ be a semi-Hilbert space of holomorphic functions.  Assume that the following hold:
	\begin{itemize}
		\item  the canonical mapping $H\to \Hol(\Uc_{n+1})/\ker \partial_2^k$ is continuous and non-trivial;
		
		\item $U_s$ induces a bounded (resp.\ isometric) representation of $G_T$ in $H$.
	\end{itemize} 
	Then, $s+2k\Meg 0$, $	H\subseteq \Ac_{s,k}$ continuously,	and the canonical mapping $H/(H\cap \ker \partial_2^k )\to \widehat \Ac_{s,k} $ is an isomorphism (resp.\ a multiple of an isometry). 
\end{cor}

We first need the following lemma.

\begin{lem}\label{lem:1}
	Take $s\in \R$ and $k\in \N$. Then, for every $f\in \Hol(\Uc_{n+1})$ and for every $\phi\in \Aff(\Uc_{n+1})$,
	\[
	(U_{s+2k}(\phi))\partial_2^k f=\partial_2^k (U_{ s}(\phi) f).
	\]
\end{lem}

\begin{proof}
	The assertion is clear if $\phi\in \Hb_n$ and also if $\phi\in U(n)$ (acting on the fisrt variable). Then, assume that $\phi\colon (\zeta,z)\mapsto (R\zeta,R^2z)$ for some $R>0$. Then,
	\[
	\partial_2^k(f\circ \phi^{-1})=R^{-2 k} (\partial_2^k f)\circ \phi^{-1}
	\]
	so that the assertion follows, since $R^{- 2k}=  
        \abs{(J\phi^{-1})^{2k/(n+2)}(0,i)} $.
\end{proof}

\begin{proof}[Proof of Corollary~\ref{cor:1}]
	Observe that the mapping $\partial_2^k\colon \Hol(\Uc_{n+1})/\ker \partial_2^k\to \Hol(\Uc_{n+1})$ is an isomorphism (onto), thanks to~\cite[Theorem 9.4]{Treves}.
	On the one hand, this implies that $\partial_2^k$ maps $\Ac_{s,k}$ onto $\Ac_{s+2k}$, so that $\Ac_{s,k}$ is complete. The invariance of $\Ac_{s,k}$ then follows from Theorem~\ref{teo:1} and Lemma~\ref{lem:1}.
	
	On the other hand, this implies that the canonical mapping $H\to \Hol(\Uc_{n+1})/\ker \partial_2^k$ is continuous and non-trivial if and only if the linear mapping $\partial_2^k\colon H\to \Hol(\Uc_{n+1})$ is continuous and non-trivial. Therefore, using Lemma~\ref{lem:1}, we see that $\partial_2^k(H)$ is a Hilbert space which embeds continuously into $\Hol(\Uc_{n+1})$ and in which $U_{s+2 k}$ induces a bounded (resp.\ unitary) representation of $G_T$. Then, Theorem~\ref{teo:1} implies that  $\partial_2^k(H)=\Ac_{s+2 k}$ with an equivalent (resp.\ proportional) norm, whence the result.
\end{proof}

\section{Invariant Spaces for $n=0$}\label{sec:4bis}

In this section, we assume that $n=0$, and we characterize $U_s$-  and $\widetilde U_s$-invariant strongly decent complete prehilbertian spaces on $\Uc_1=\C_+$. 

We begin by describing the closed $\Aff(\C_+)$- and $\widetilde U_s$-invariant subspaces of $\Hol(\C_+)$. We denote by $\Pc^k(\C)$, or simply $\Pc^k$, the space of holomorphic polynomials on $\C$ of degree $<k$.

Observe  that a vector subspace of $\Hol(\C_+)$ is
$\Aff(\C_+)$- $U_s$-invariant if and only if it is
 $\Aff(\C_+)$-$U_0$-invariant, so that we shall simply say 
$\Aff(\C_+)$-invariant in this case.

 We observe explicitly that here and in Section~\ref{sec:5} we shall first characterize the closed $\Aff$- and $\widetilde U_s$-invariant vector subspaces of $\Hol$. This classification will be fundamental in the characterization of $U_s$-invariant semi-Hilbert spaces of holomorphic functions, and may actually be interpreted as a description of certain mean-periodic functions in $\Hol$. Since the proof in the case $n>0$ is not simpler to the one needed for general Siegel domains (of type II), we shall defer the proof of these results until Section~\ref{sec:app}, where the general case is handled for future reference. Here and in Proposition~\ref{prop:30} we shall limit ourselves to specializing the general results in this context, as well as providing a more precise description of the resulting spaces.

\begin{prop}\label{prop:28}
	Let $V$ be a \emph{proper} closed subspace of
        $\Hol(\C_+)$ and take $s\in \R$. Then, $V$ is  
        $\Aff(\C_+)$-invariant if and only if $V=\Pc^k$ for some $k\in \N$.
	
	In addition, $V$ is $\widetilde \Aut$-$\widetilde U_s$-invariant if and only if $V=\Set{0}$ or $s\in -\N$ and $V=\Pc^{1-s}$.
\end{prop}

\begin{proof}
	By Proposition~\ref{prop:6}, if $V$ is $\Aff(\C_+)$-invariant, then $V\cap \Pc$ is dense in $V$, where $\Pc$ denotes the space of holomorphic polynomials on $\C$. Since  the $\Aff(\C_+)$-invariant  proper subspaces of $\Pc$ are clearly the $\Pc^k$, $k\in\N$, the first assertion follows.
	
	Now, take $k\in\N$ and observe that, by Proposition~\ref{prop:27}, $\Pc^k$ is  $\widetilde \Aut$-$\widetilde U_s$-invariant if and only if $\widetilde U_s(\iota)\Pc^k= \Pc^k$, where $\iota\colon z \mapsto - 1/z$.\footnote{Notice that this action does \emph{not} determine $\iota$ in $\widetilde \Aut$, but determines its image in $G$, which is sufficient to determine $\widetilde U_s(\iota)$ up to a unimodular constant.} Now, observe that, for every $h=0,\dots, k-1$,
	\[
	\widetilde U_s(\iota)(\,\cdot\,)^h=(-1)^h(\,\cdot\,)^{-h-s}
	\]
	thanks to~\eqref{eq:4}, possibly up to a unimodular constant. Therefore, in order that $\Pc^k$ be $\widetilde \Aut$-$\widetilde U_s$-invariant, it is necessary and sufficient that $\Set{0,\dots, k-1}=\Set{-s,\dots,-(k-1)-s}$, that is, $k=1-s$. The proof is complete.
\end{proof}

Now, we shall indicate an intertwining formula between $\widetilde
U_s$ and  $\widetilde U_{2-s}$, when $s\in-\N$.

\begin{prop}\label{prop:29}
	Take $s\in -\N$. Then, 
	\[
	[\widetilde U_s(\phi) f]^{(1-s)}=\widetilde U_{2-s}(\phi)
        f^{(1-s)}
	\]
	for every $\phi\in \widetilde \Aut$ and for every $f\in \Hol(\C_+)$.
\end{prop}

When $s=0$, this is a particular case of~\cite[Theorem 3.1]{Garrigos}. See~\cite[Theorem 6.4]{Arazy} for the case of the unit disc (and more general symmetric domains).

\begin{proof}
	Notice first that Lemma~\ref{lem:1} implies that 
	\[
	[\widetilde U_s(\phi) f]^{(1-s)}=\widetilde U_{2-s}(\phi) f^{(1-s)}
	\]
	for every $\phi\in G_T$ 
         (identified with a subgroup of
        $\widetilde \Aut$), and for every $f\in \Hol(\C_+)$. The same
        then holds for every $\phi\in \widetilde \Aut$ whose canonical
        image in $\Aut(\C_+)$ belongs to $G_T$. 
	Then, take $\iota\in \widetilde \Aut$ such that $\iota\colon z \mapsto - 1/z$, and observe that there is $\eps\in \Set{\pm1}$ such that
	\[
	\widetilde U_{s'}(\iota)f(z)= (\eps z)^{-s'} f(-1/z) 
	\]
	for every $f\in \Hol(\C_+)$, for every $s'\in -\Z$, and for every $z\in \C_+$. Then,  for every $h\in \N$, 
	\[
	[\widetilde U_s(\iota) (\,\cdot\,)^h]^{(1-s)}(z)=\eps^{-s} (-1)^h (-h-s)(-h-s-1)\cdots (-h) z^{-h-1}
	\]
	and
	\[
	\widetilde U_{2-s}(\iota) [(\,\cdot\,)^h]^{(1-s)}(z)=\eps^{s-2} (-1)^{h-1+s} h (h-1)\cdots(h+s) z^{-h-1 }
	\]
	for every $z\in \C_+$, whence the asserted equality in this
        case. Since the space of polynomials is dense in $\Hol(\C_+)$,
    the assertion follows thanks to Proposition~\ref{prop:27}.
\end{proof}

\begin{teo}\label{teo:6}
	Take $s\in \R$. If $k\in \N$ and $s+2 k\Meg 0$, then
        $\Ac_{s,k}$ is strongly decent, saturated, and         $\Aff(\C_+)$-$U_s$-invariant   with its seminorm.
	If, in addition, $s\Meg 0$ and $k=0$, or $s\in-\N$ and
        $k=1-s$, then   $\Ac_{s,k}$ is   $\widetilde \Aut$-$\widetilde
        U_s$-invariant  with its seminorm.
	
	Conversely, let $H$ be decent semi-Hilbert space of holomorphic functions such that $U_s$ induces a bounded (resp.\ isometric) representation of $\Aff(\C_+)$ in $H$. Assume that one of the following conditions hold:
	\begin{enumerate}
		\item[\textnormal{(a)}] $H$ is strongly decent;
		
		\item[\textnormal{(b)}] $\norm{(\,\cdot\,)^{-s/2}}_H=0$ if $s\in -2 \N$ and $(\,\cdot\,)^{-s/2}\in H$.
	\end{enumerate}
	Then,  there is $k\in \N$ such that $s+2k\Meg 0$ and $H$ embeds as a dense subspace of $\Ac_{s,k}$ with  the induced topology (resp.\ and proportional seminorms). If, in addition, $\widetilde U_s(\iota)$ induces a continuous automorphism of $H$ for some $\iota\in \widetilde \Aut$ such that $\iota\colon z \mapsto -1/z$,  then   either $k=0$ or $s\in -\N$ and $k=1-s$.
\end{teo}

Notice that saying that $H$ embeds as a dense subspace in $\Ac_{s,k}$ (with the induced topology) means that $H$ is a subspace of $\Ac_{s,k}$  such that  $H+\ker (\,\cdot\,)^{(k)} =H+\Pc^k =\Ac_{s,k}$, thanks to the completeness of $H$.

We observe explicitly that we could remove assumptions (a) and (b) and
show that $H=\Ac_{s,k}$, if we knew that the only dense
$\Aff(\C_+)$-invariant vector subspaces of $\Ac_{s,k}$, for $s+2k>0$, are precisely the $\Ac_{s,k'}$ for $s+2k'>0$.
Nontheless, we do not even known if the space $\Pc^k$ admits 
an $\Aff(\C_+)$-invariant   algebraic complement $H'$ in
$\Ac_{s,k}$ when $s+2 k>0$. Notice, though, that if such a space $H'$
exists, then $H'$ is a Hilbert space, but does \emph{not} embed
continuously into $\Hol(\C_+)$. Therefore, such spaces, if existing,
must be somewhat pathological. 

We point out that 
the statements of~\cite[Theorems 1 and 2]{ArazyFisher} (which deal with  the case $s=0$) 
seem to be erroneous: on the one hand,  the null space and the space
of constant functions satisfy all the assumptions of the cited
results, but are not the Dirichlet space; on the other hand, the fact
that a M\"obius-invariant space $H$ is contained in the Dirichlet
space and has the induced topology (or even the induced seminorm) does
\emph{not} immediately imply that it is the whole Dirichlet space, as
there  might  exist M\"obius-invariant \emph{algebraic} complements of
the space of constant functions in the Dirichlet space (and this
possibility is not ruled out in the proof).\footnote{We point out,
  though, as we remarked earlier, that we do \emph{not} know whether
  such pathological spaces do exist or not.} 

We also point out, in this connection, that~\cite[Remark]{ArazyFisher} seems to be erroneuos: on the one hand, one must assume that $H$ contains non-constant functions in order to be sure that the given arguments produce non-zero polynomials in $H$; on the other hand, unless one assumes that $H$ contains the constant functions, it is only possible to deduce that the non-constant monomials belong to $H$, and this does not seem to be sufficient to conclude that $H$ cannot be an algebraic complement of the space of constant functions in the Dirichlet space. The subsequent proof then builds on the erronous claim that, if $f(z)=\sum_k a_k z^k$ (Taylor development about $0$) and $f\in H$, then $f=\sum_k a_k (\,\cdot\,)^k$ (orthogonal sum in $H$): despite the fact that the $(\,\cdot\,)^k$, $k\in N$, are pairwise orthogonal in $H$ (where $N$ is the set of $k\in \N$ such that $(\,\cdot\,)^k\in H$, so that $N$ is either $\N$ or $1+\N$), one may not claim that the Taylor development of the orthogonal sum $\sum_{k\in N} a_k(\,\cdot\,)^k$ must be $\sum_{k\in N} a_k z^k$ -- unless the sequence $(a_k)$ has finite support -- since $H$ does \emph{not} necessarily embed continuously in $\Hol(\C_+)$ (but only in $\Hol(\C_+)$ modulo constant functions).

\begin{proof}
The first assertion follows from Corollary~\ref{cor:1} and Proposition~\ref{prop:29}.

Then, take $H$ as in the statement.
By Proposition~\ref{prop:31}, there is an $\Aff(\C_+)$-invariant closed vector subspace $V$ of $\Hol(\Uc_{n+1})$ so that the canonical mapping $H\to \Hol(\Uc_{n+1})/V$ is continuous and non-trivial. Observe that Proposition~\ref{prop:28} implies that there are $h,k\in\N$, $h\meg k$, such that $V=\Pc^k$ and $\Pc^h$ is the ($\Aff(\C_+)$-invariant) closure of $\Set{0}$ in $H$. 
Therefore, Corollary~\ref{cor:1} implies that $H\subseteq \Ac_{s,k}$ continuously, and that the continuous linear mapping $H/(\Pc^k\cap H)\to  \Ac_{s,k}/\Pc^k$ is an isomorphism.  In order to complete the proof of the first assertion, it then suffices to show that $\Pc^k\cap H=\Pc^h$. If condition (a) holds, that is, if $H$ is strongly decent, then Proposition~\ref{prop:31} ensures that we may take $k$ so that this is the case.

Then, assume that condition (b) holds, that is, that $\norm{(\,\cdot\,)^{-s/2}}_H=0$ (i.e., $-s/2<h$) when $s\in -2 \N$ and $(\,\cdot\,)^{-s/2}\in H$. Observe first that $\Pc^k\cap H$ is finite-dimensional and $\Aff(\C_+)$-invariant, so that it equals $\Pc^\ell$ for some $\ell \in \Set{h,\dots,k}$, thanks to Proposition~\ref{prop:28} again. 
In particular, $U_s$ induces a bounded representation in $\Pc^\ell/\Pc^h$, so that $U_s(R\,\cdot\,)((\,\cdot\,)^j)=R^{-j-s/2}(\,\cdot\,)^{j}$  must stay bounded as $R$ runs through $\R_+^*$, for every $j=h,\dots, \ell-1$. Therefore, either $\ell=h$ or $\ell-1=h=-s/2$. Now, if $\ell-1=h=-s/2$, then $(\,\cdot\,)^{-s/2}\in H$, so that $-s/2< h$ by (b): contradiction.  Thus, $\ell=h$, as we wished to prove.

Finally, assume that $\widetilde U_s(\iota)$ induces a continuous endomorphism of $H$ for some $\iota\in \widetilde \Aut$ such that $\iota\colon z \mapsto -1/z$. Observe that, in this case, $\Pc^k$ and $\Pc^h$ must be  $\widetilde \Aut$-$\widetilde U_s$-invariant (cf.~Proposition~\ref{prop:27}), so that Proposition~\ref{prop:28} implies that either $k=h=0$, or $s\in -\N$ and $k=1-s $ (and $h\in \Set{0,1-s}$). The proof is therefore complete.
\end{proof}

\begin{oss}\label{oss:4}
	If $s,s'\in \R$ and $k,k'\in \N$ are such that $s+2k, s'+2k'>0$, then $U_s$ and $U_{s'}$, as unitary representations of $\Aff(\C_+)$ into $\widehat\Ac_{s,k}$ and $\widehat\Ac_{s',k'}$, respectively, are unitarily equivalent. 
	
	If $(s,k),(s',k')\in( \R_+^*\times \Set{0})\cup \Set{(-h,1-h)\colon h\in \N}$, then the unitary representations $\widetilde U_s$ and $\widetilde U_{s'}$ of  $\widetilde \Aut(\C_+)$ into   $\widehat\Ac_{s,k}$ and $\widehat\Ac_{s',k'}$, respectively, are unitarily equivalent if and only if $s=s'$, $s'=s+2k$, or $s=s'+2k'$.
\end{oss}

The first assertion obviously extends to the case $n>0$. The second assertion does not, as we shall see in Remark~\ref{oss:3}.

\begin{proof}
	The first assertion follows from Remark~\ref{oss:2}. If $s,s'>0$ or $s,s'\in -\N$, then also the second assertion follows from Remark~\ref{oss:2}. The remaining cases follow from Proposition~\ref{prop:29} and Remark~\ref{oss:2}.
\end{proof}

\section{Invariant Spaces for $n>0$}\label{sec:5}

In this section, we assume that $n>0$, and we characterize $U_s$- and
$\widetilde U_s$-invariant strongly decent semi-Hilbert
spaces of holomorphic functions on $\Uc_{n+1}$.  For simplicity of notation, we write $\Uc$ in place of $\Uc_{n+1}$.

As in Section~\ref{sec:4bis}, we shall simply say that a vector subspace of $\Hol(\Uc)$ is $\Aff(\Uc)$-invariant if it is $\Aff(\Uc)$-$U_s$-invariant for some $s\in \R$.

We begin by describing the closed $\Aff(\Uc)$-  and $\widetilde U_s$-invariant subspaces of $\Hol(\Uc)$.
  Recall that $\Pc^k$ denotes the space of holomorphic polynomials of degree $<k$. We denote by $\Pc_k$ the space of \emph{homogeneous} holomorphic polynomials  of degree $k$. 

\begin{prop}\label{prop:30}
	Let $V$ be a \emph{proper} closed subspace of
        $\Hol(\Uc)$, and take $s\in \R$. Then, $V$ is 
        $\Aff(\Uc)$-invariant  if and only if $V$ is the closure of $\bigoplus_{k\in \N} [\Pc_k(\C^n)\otimes \Pc^{h_k}(\C)]$ in $\Hol(\Uc)$, where $h_0\in \N$ and $h_{k+1}\in \Set{h_k, (h_k-1)_+}$ for every $k\in \N$.
	
	In addition, $V$ is $\widetilde \Aut$-$\widetilde U_s$-invariant        if and only if $V=\Set{0}$ or $s\in -\N$ and  $V=\Pc^{1-s}(\C^{n+1})$.
\end{prop}

\begin{proof}
	\textsc{Step I.} By Proposition~\ref{prop:6}, we may reduce to describing $V\cap \Pc$, where $\Pc$ is the space of holomorphic polynomials on $\Uc$. 
	
	Observe that, if $\mi$ is a (Radon) measure with compact support on 
        $\Aff(\Uc)$, then we may define
\begin{equation}\label{def:Usmi}
  U_s(\mi)\coloneqq \int_{\Aff(\Uc)} U_s(\phi)\,\dd \mi(\phi)
\end{equation}
as a weak integral in $\Lin(\Hol(\Uc))$, endowed with the  topology of pointwise convergence. Note that this operator is independent of $s$ if $\mi$ is supported in $\Nc U(n)$. 
	Now, fix $k\in\N$ and define a measure $\mi_k$ on
        $\Aff(\Uc)$  as follows. Consider the canonical homomorphism $\pi\colon\T\ni \alpha \mapsto [(\zeta,z)\mapsto (\alpha \zeta,z)]\in \Aff(\Uc)$, and define
	\[
	\langle \mi_k, f\rangle\coloneqq  \int_{\T} \alpha^{k} f(\pi(\alpha))\,\dd \alpha
	\]
	for every $f\in C_c(G_\Aff)$. 
	Then, consider the operator
\begin{equation}\label{def:pik}
	\pi_k\coloneqq U_{\lambda }(\mi_k),
\end{equation}
	and observe that
	\[
	\pi_k(f)(\zeta,z)=\frac{1}{k!} \partial_1^k f(0,z)\cdot \zeta^k
	\]
	for every $f\in \Hol(\Uc)$ and for every $(\zeta,z)\in \Uc$.  It is then clear that $\pi_k$ maps $\Hol(\Uc)$ continuously onto the space $\Pc_k(\C^n)\otimes \Hol(\C_+)$, canonically identified with a subspace of $\Hol(\Uc)$.
	Furthermore,~\cite[Theorem 2.1]{FarautKoranyi2} implies that $U_{s }$ induces an  irreducible representation $U'_k$ of the unitary group $U(n)$ in $\Pc_k(\C^n)$, for every $k\in\N$, and that these representations are pairwise  inequivalent. 
	Let us prove that, for every $k\in \N$, there is a vector subspace $V_k$ of $\Hol(\C_+)$  such that $\pi_k(V)=\Pc_k(\C^n)\otimes V_k$.\footnote{  Here we are interested in showing that $\pi_k(V)$ is the tensor product of $\Pc_k(\C^n)$ with some closed vector subspace of $\Hol(\C_+)$, with reference to the canonical embedding of $\Pc_k(\C^n)\otimes \Hol(\C_+)$ into $\Hol(\Uc)$, and not simply the tensor product of $\Pc_k(\C^n)$ with some abstract vector space (as the general theory of representation shows).}
	Observe first that $U'_k(\cM(U(n)))$ is a unital sub-$C^*$-algebra of $\Lin(\Pc_k(\C^n))$, and that, by Schur's lemma, its commutant is $\C I$ (cf.~\cite[Theorem 2 of \S\ 21, No.\ 3]{Naimark}). Therefore, by the bicommutant theorem (cf.~\cite[Theorem of \S\ 34, No.\ 2]{Naimark}), $U'_k(\cM(U(n)))=\Lin(\Pc_k(\C^n))$. 
	Then, take a  basis $p_1,\dots, p_h$ of $\Pc_k(\C^n)$ and $f\in \pi_k(V)$. Observe that there are uniquely determined elements $f_1,\dots, f_h$ of $\Hol(\C_+)$ such that $f=\sum_j p_j\otimes f_j$. Take $j_0\in \Set{1,\dots,h}$ and $\mi_{j_0}\in \cM(U(n))$ such that $U'_k(\mi_{j_0})p_j=\delta_{j,j_0} p_j$  for every $j=1,\dots, h$. Then,
	\[
	V\ni U_s(\mi_{j_0}) f=\sum_j U'_k(\mi_{j_0})(p_j\otimes f_j)= p_{j_0}\otimes f_{j_0}.
	\]
	By the arbitrariness of $f$ and $j_0$, this implies that $\pi_k(V)=\Pc_k(\C^n)\otimes V_k$ for some vector subspace $V_k$ of $\Hol(\C_+)$. Observe that, since the $\pi_k$ are projectors of $\Hol(\Uc)$, $\pi_k \pi_h=\pi_h\pi_k=0$ for $h\neq k$, and $\sum_k \pi_k=I$ in the strong topology of $\Lin(\Hol(\Uc))$, it is clear that $\pi_k(V)=V\cap \big(\bigcap_{k'\neq k}\ker \pi_{k'}\big)$ is closed in $\Hol(\Uc)$, so that $V_k$ is closed in $\Hol(\C_+)$. In addition, it is clear that
	\begin{equation}\label{eq:6}
	U_s(\phi) (p\otimes f)= p\otimes U_{s+ k}(\phi') f
	\end{equation}
	for every $p\in \Pc_k(\C^n)$, for every $f\in \Hol(\C_+)$,   for every $\phi \in \Aff(\Uc)$, and for every $\phi'\in \Aff(\C_+)$ such that  $\phi\colon (\zeta,z)\mapsto (\zeta,z+x)$ and $\phi'\colon z \mapsto z+x$ for some $x\in \R$, or $\phi\colon (\zeta,z)\mapsto (R^{1/2}\zeta,R z)$ and $\phi'\colon z\mapsto R z$ for some $R>0$. 
	It then follows that $V_k$ is $\Aff(\C_+)$-invariant, so that it equals  $\Pc^{k_h}(\C)$ for some $k_h\in \N\cup \Set{\infty}$, interpreting $\Pc^\infty(\C)$ as the space of holomorphic polynomials on $\C$, thanks to Proposition~\ref{prop:28}. Thus, $V\cap \Pc=\bigoplus_k [\Pc_k(\C^n)\otimes \Pc^{h_k}(\C)]$. 
	
	\textsc{Step II.} Take $k,h\in\N$, and let us prove that the $ \Aff(\Uc)  $-invariant vector subspace $V_{k,h}$ of $\Pc$ generated by $\Pc_k(\C^n)\otimes \Pc^h(\C)$ is  
	\[
	\Pc_{k,h}\coloneqq \sum_{k'\meg k, h_1+h_2=h} [\Pc_{k'+h_1}(\C^n)\otimes \Pc^{h_2}(\C)]=\bigoplus_{\ell=0}^{h+k-1} [\Pc_\ell(\C^n)\otimes \Pc_{h-(\ell-k)_+}(\C)].
	\]
	Observe that, since $\Aff(\Uc) $ is the semi-direct product of $GL(\Uc)$ and $\Hb_n$ (cf.~Subsection~\ref{sec:2:1}), and since  $\Pc_k(\C^n)\otimes \Pc^h(\C)$ is clearly $GL(\Uc)$-invariant, it will suffice to consider the action of $\Hb_n$.
	Then, take $p\in \Pc_k(\C^n)$ and $h'<h$. Observe that, for every $(\zeta,x)\in \Hb_n$,
	\[
	\begin{split}
	U_s(\zeta,x) (p\otimes (\,\cdot\,)^{h'})(\zeta',z')&= p(\zeta'-\zeta)(z'-x+i\abs{\zeta}^2-2 i \langle \zeta'\vert \zeta\rangle)^{h'}\\
		&=\sum_{k'=0}^k \sum_{h_1+h'_2=h'} \frac{h'!}{h_1!h'_2! k'!} p^{(k')}(-\zeta) \cdot \zeta'^{k'} (-2i\langle \zeta'\vert \zeta\rangle)^{h_1}(z'-x+i\abs{\zeta}^2)^{h'_2}
	\end{split}
	\]
	for every $(\zeta',z')\in \Uc$, so that $U_s(\zeta,x) (p\otimes (\,\cdot\,)^{h'})\in \Pc_{k,h}$.
	It is then readily seen that  $V_{k,h}\subseteq \Pc_{k,h}$. Conversely, choose $p=\langle \,\cdot\,\vert \zeta \rangle^k$ in the above computations, and observe that
	\[
	U_s(\zeta,x) (p\otimes (\,\cdot\,)^{h'})(\zeta',z')=\sum_{k'=0}^k \sum_{h_1+h'_2=h'} (-2i)^{h_1}\frac{h'!k!}{h_1!h'_2! k'!(k-k')!} (-\abs{\zeta}^2)^{k-k'} \langle \zeta'\vert \zeta\rangle^{k'+h_1}  (z'-x+i\abs{\zeta}^2)^{h'_2}
	\]
	for every $(\zeta',z')\in \Uc$. Therefore, for every $\ell\meg h+k$, $\pi_\ell(U_s(\zeta,x) (p\otimes (\,\cdot\,)^{h'}))(\zeta',z')$ equals
	\[
	\langle \zeta'\vert \zeta\rangle^{\ell}  \sum_{k'=(\ell-h')_+  }^{\min(k,\ell)}  (-2i)^{\ell-k'}\frac{h'!k!}{(\ell-k')!(h'-\ell+k')! k'!(k-k')!} (-\abs{\zeta}^2)^{k-k'}  (z'-x+i\abs{\zeta}^2)^{h'-\ell+k'}.
	\]
	Since the highest power (in $z'$) in the above sum occurs only once (for $k'=\min(k,\ell)$) and with a non-zero coefficient, by the arbitrariness of $x$ it is clear that   $(\zeta',z')\mapsto \langle \zeta'\vert \zeta\rangle^{\ell}  z'^{h_2'}$  belongs to $V_{k,h}$ for every  $h_2'\meg h'-(\ell-k)_+$. Thus, $V_{k,h}=\Pc_{k,h}$. 
	
	\textsc{Step III.} By~\textsc{step I},  we know that $V\cap \Pc=\sum_{k\in \N}[\Pc_k(\C^n)\otimes \Pc^{h_k}(\C)]$. By~\textsc{step II}, we know that $\Pc_{k'}(\C^n)\otimes \Pc^{h_k}(\C)\subseteq V$ for every $k'\meg k$, so that $h_{k'}\Meg h_k$ and $(h_k)$ is decreasing, and that $\Pc_{k+\ell}(\C^n)\otimes \Pc^{h_k-\ell}(\C)\subseteq V$ for every $k$ and for every $\ell<h_k$,  so that $h_{k+\ell}\Meg h_k-\ell $.  Consequently, $h_0\in \N$\footnote{If $h_0=\infty$, then $h_k=\infty$ for every $k\in\N$, and $V=\Hol(\Uc)$, contrary to our assumptions.} and $h_{k+1}\in \Set{h_k, (h_k-1)_+}$ for every $k\in\N$. Conversely, given a sequence $(h_k)$ with the preceding properties, the space $\bigoplus_k [\Pc_k(\C^n)\otimes \Pc^{h_k}(\C)]$ is $ \Aff(\Uc) $-invariant, since it contains $\Pc_{k,h_k}$ for every $k\in\N$, with the notation of~\textsc{step II}. This completes the proof of the first assertion.
	
	\textsc{Step IV.} Now, take $V$ as the closure of  $\bigoplus_k[ \Pc_k(\C^n)\otimes \Pc^{h_k}(\C)]$ in $\Hol(\Uc)$, where $h_0\in \N$ and $h_{k+1}\in \Set{h_k,(h_k-1)_+}$ for every $k\in\N$. Observe that we may take $\iota\in \widetilde \Aut$ so that $\iota\colon (\zeta,z)\mapsto (-i\zeta/z,-1/z)$ and so that
	\[
	\widetilde U_s(\iota) f(\zeta,z)=i^{-n s/(n+2)} f(\iota(\zeta,z)) z^{-s} 
	\]
	for every $f\in \Hol(\Uc)$ and for every $(\zeta,z)\in \Uc$, thanks to Proposition~\ref{prop:27} and~\eqref{eq:4}. In addition, $V$ is $\widetilde \Aut$-$\widetilde U_s$-invariant if and only if $\widetilde U_s(\iota) V\subseteq V$, thanks to Proposition~\ref{prop:27} again. Then take $p\in \Pc_k(\C^n)$ and $f\in \Hol(\C_+)$, and observe that
	\begin{equation}\label{eq:7}
	\widetilde U_s(\iota) (p\otimes f)= i^{-n s/(n+2)} (-i)^k p\otimes \widetilde U_{s+k}(\iota_0)f,
	\end{equation}
	where $\iota_0$ is a suitable element of $\widetilde \Aut(\C_+)$ such that $\iota_0\colon z \mapsto -1/z$.
	Taking~\eqref{eq:6} into account, this proves that $V$ is $\widetilde \Aut$-$\widetilde U_s$-invariant if and only if $\Pc^{h_k}(\C)$ is $\widetilde \Aut(\C_+)$-$\widetilde U_{s+k}$-invariant for every $k\in \N$. 
	If $V$ is  $\widetilde \Aut$-$\widetilde U_s$-invariant, then  Proposition~\ref{prop:28} implies that either $h_k=0$ for every $k\in\N$, or $s\in -\N$ and $h_k=(1-s-k)_+$ for every $k\in\N$. Conversely,~\eqref{eq:7} and~\textsc{step III} show that  $ \Pc^{1-s}(\C^{n-1})= \bigoplus_k [\Pc_k(\C^n)\otimes \Pc^{(1-s-k)_+}(\C)]$ is $\widetilde \Aut$-$\widetilde U_s$-invariant.	
\end{proof}

\begin{teo}\label{teo:3}
	Take $s\in \R$ and let $H$ be a non-trivial strongly decent semi-Hilbert space of holomorphic functions on $\Uc$ such that $U_{s}$ induces a bounded (resp.\ isometric) representation of $\Aff(\Uc)$ in $H$. 
	Then, either one of the following conditions hold:
	\begin{itemize}
		\item there is $k\in\N$ such that $s+2 k>0$ and $H$ is a dense subspace of $\Ac_{s,k}$ with an equivalent (resp.\ proportional) seminorm;
		
		\item $s\in -\N$ and there is a non-empty subset $J$ of $\N\cap (-s-2 \N)$ such that $H$ contains $\bigoplus_{j\in J}[\Pc_j(\C^n)\otimes \Pc^{-(s+j)/2}(\C)]$ as a dense vector subspace, and thereon its seminorm is equivalent (resp.\ equal) to the  seminorm 
        \[
		\norm*{ \sum_{j\in J} (p_j\otimes (\,\cdot\,)^{-(s+j)/2})}^2=\sum_{j\in J} \norm{p_j}^2_{j}
		\]
		for every $\sum_{j\in J}( p_j\otimes (\,\cdot\,)^{-(s+j)/2})\in \bigoplus_{j\in J}[\Pc_j(\C^n)\otimes \Pc^{-(s+j)/2}(\C)]$, where $\norm{\,\cdot\,}_j$ is a suitable rotation-invariant Hilbert norm on the space  $\Pc_j(\C^n)$. 
              \end{itemize}
\end{teo}

Notice that Corollary~\ref{cor:1} shows that $\Ac_{s,k}$ is $
\Aff(\Uc)  $-$U_s$-invariant whenever $s+2k\Meg 0$. The fact that
the remaining spaces (of polynomials) described above are actually
$GL(\Uc)$-$U_s$-invariant (hence generate finite-dimensional
$\Aff(\Uc) $-$U_s$-invariant spaces of polynomials) is a simple verification left to the reader. We do not describe all the resulting invariant spaces, but we observe that all non-empty subsets $J$ of $\N\cap (-s-2\N)$ may occur, as one  may see by means of Proposition~\ref{prop:30}.

  Notice that the closure of $\Set{0}$ in $H$ may be determined by means of Proposition~\ref{prop:30}, at least when $H$ is saturated. 

\begin{proof}
	  Fix a $U(n)$-invariant Hilbert norm on $\Pc_k(\C^n)$ for every $k\in\N$.  
	
	By Proposition~\ref{prop:31} and the remark following its
        statement, we may assume that $H$ is saturated, so that the closure $V$ of $\Set{0}$ in
        $H$ is closed in $\Hol(\Uc)$, and that the canonical
        linear mapping $H\to \Hol(\Uc)/V$ is continuous.  In addition,
        Proposition~\ref{prop:32} shows that $H$ is separable and that
        $U_s$ induces a \emph{continuous} representation of $\Aff(\Uc)$ in $H$. Consequently, if $\mi$ is a (Radon)
        measure with compact support on $ \Aff(\Uc) $, then we
        may define $U_s(\mi)$ as in~\eqref{def:Usmi} both as an endomorphism $U_1$ of $\Hol(\Uc)$ and as an endomorphism $U_2$ of $H/V$. By the continuity of the mapping $H\to \Hol(\Uc)/V$, it is clear that $U_1$ induces $U_2$.
        
          Further, observe that we may assume that the seminorm of $H$ is $U(n)$-invariant, up to replace it with the seminorm induced by the equivalent scalar product $(f,g) \mapsto \int_{\phi\in U(n)}\langle U_s(\phi) f\vert U_s(\phi)g\rangle_H\,\dd \phi$.  
	
	Then, define $\pi_k$ as in~\eqref{def:pik} and observe that, by our additional
        assumption, the $\pi_k(H)$ are pairwise orthogonal subspaces
        of $H$, and $\pi_k(H)\cap \pi_{k'}(H)=V$ for every $k\neq
        k'$. In addition, arguing as in the proof of
        Proposition~\ref{prop:30}, we see that for every $k\in\N$
        there is a vector subspace $H_k$ of $\Hol(\C_+)$ such that
        $\pi_k(H)=\Pc_k(\C^n)\otimes H_k$. In addition, by
        Proposition~\ref{prop:30}, there is a sequence $(h_k)$ such
        that $h_0\in \N$ and $h_{k+1}\in \Set{h_k,(h_k-1)_+}$ for
        every $k\in\N$, and such that $V$ is the closure of
        $\bigoplus_k [\Pc_k(\C^n)\otimes \Pc^{h_k}(\C)]$ in
        $\Hol(\Uc)$.
	
	Now, for every $f\in H_k\setminus \Pc^{h_k}(\C)$, it is clear that $\Pc_k(\C^n)\otimes \C f$ is a $U(n)$-$U_s$-irreducible subspace of $\pi_k(H)$. 
	In addition, for every vector subspace $H'$ of $H_k$, the orthogonal complement of $\Pc_k(E)\otimes H'$ in $\pi_k(H)$ is of the form $\Pc_k(\C^n)\otimes H''$, for some vector subspace $H''$ of $H_k$ (containing $\Pc^{h_k}(\C)$).  
	Consequently, we may find a (possibly empty) sequence $(f_j)$ of elements of $H_k\setminus \Pc^{h_k}(\C)$ such that $\pi_k(H)$ is the orthogonal direct sum of $\Pc_k(\C^n)\otimes \Pc^{h_k}(\C)$ and the $\Pc_k(\C^n)\otimes \C f_j$. 
	If we choose the $f_j$ so that $\norm{p\otimes f_j}_H=1$ for every unit vector $p$ in $\Pc_k(\C^n)$ (this is possible since the seminorm of $H$ is $U(n)$-invariant), and we define a prehilbertian seminorm on $H_k$ so that its null space is $\Pc^{h_k}(\C)$ and $(f_j+\Pc^{h_k}(\C))$ is an orthonormal basis of $H_k/\Pc^{h_k}(\C)$, then clearly $H_k$ is complete and $\pi_k(H)=\Pc_k(\C^n)\otimes_2 H_k$.\footnote{Here, we denote by $\Pc_k(\C^n)\otimes_2 H_k$ the tensor product of the semi-Hilbert spaces $\Pc_k(\C^n)$ and $H_k$, endowed with the scalar product defined by $\langle p\otimes f\vert q\otimes g \rangle\coloneqq \langle p\vert q\rangle\langle f \vert g \rangle$ for every $p,q\in \Pc_k(\C^n)$ and for every $f,g\in H_k$. }
	
	Let us then prove that $H_k$ is strongly decent and saturated, that is, that the linear mapping $H_k\to \Hol(\C_+)/\Pc^{h_k}(\C)$ is continuous. Observe that  the mapping
	\[
	\Pc_k(\C^n)\otimes_2 H_k\to \Pc_k(\C^n)\otimes (\Hol(\C_+)/\Pc^{h_k}(\C))
	\]
	is continuous.\footnote{If $(f_j)$ is a sequence in $\Pc_k(\C^n)\otimes_2 H_k$ converging to $f$, then $(f_j)$ converges to $f$ in $H$, so that $(f_j+V)$ converges to $f+V$ in $\Hol(\Uc)/V$.  Applying $\pi_k$, we then see that $(f_j+\Pc_k(\C^n)\otimes \Pc^{h_k}(\C))$ converges to $f+\Pc_k(\C^n)\otimes \Pc^{h_k}(\C)$ in $\Pc_k(E)\otimes (\Hol(\C_+)/\Pc^{h_k}(\C))$.} 
	Since the bilinear mapping $(p,f)\mapsto f$ induces continuous linear mappings
	\[
	p\otimes f\mapsto f
	\]
	of $\Pc_k(\C^n)\otimes_2 H_k$ \emph{onto} $H_k$ and of $\Pc_k(\C^n)\otimes \Hol(\C_+)$ \emph{onto} $\Hol(\C_+)$,  our claim follows   by means of the open mapping theorem.  
	
	Now, observe that~\eqref{eq:6} shows that $H_k$ is $\Aff(\C_+)$-invariant, and that  the $U_{s+k}(\phi)$, as $\phi$ runs through $\Aff(\C_+)$, are equicontinuous endomorphisms of $H_k$ (resp.\ isometries). Consequently, Theorem~\ref{teo:6} implies that either one of the following hold:
	\begin{itemize}
		\item $H_k=\Pc^{h_k}(\C)$;
		
		\item $\lambda+k\in -2 \N$, $h_k=-(s+k)/2$, and $H_k=\Pc^{-(s+k)/2+1}(\C)$, with a norm equivalent (resp.\ proportional) to the following one:
		\[
		\norm*{\sum_{j=0}^{-(s+k)/2}a_j(\,\cdot\,)^j}= \abs{a_{-(s+k)/2}};
		\]
		
		\item  $s+k+2 h_k>0$ and $H_k=\Ac_{s+ k,h_k}(\C_+)$ with an equivalent (resp.\ proportional) seminorm.
	\end{itemize} 
	Denote by $K_1,K_2,K_3$ the sets of $k\in\N$ for which the first, the second, or the third case occurs. 
	
	Assume first that  $K_3\neq \emptyset$. 
	If $\overline H$ denotes the ($\Aff(\Uc)$-invariant) closure of $H$ in $\Hol(\Uc)$, then $\pi_k(\overline H)=\Pc_k(\C^n)\otimes \Hol(\C_+)$ for every $k\in K_3$, so that Proposition~\ref{prop:30} implies that $\overline H=\Hol(\Uc)$. Consequently, $K_1=K_2=\emptyset$ (since $\pi_k(\overline H)$ is finite-dimensional for $k\in K_1\cup K_2$). 
	
	Then, $V\subseteq \ker \partial_2^{h_0}$,   and $H\neq  H\cap \ker \partial_2^{h_0}$, so that Corollary~\ref{cor:1} implies that $H\subseteq \Ac_{s,h_0}$ continuously, and that the canonical mapping $H/V\to \widehat \Ac_{s,h_0}$ is an isomorphism onto its image. In addition,
	\[
	H\cap \ker \partial_2^{h_0}=\bigoplus_{k}[\Pc_k(\C^n)\otimes \Pc^{h_0}(\C)],
	\]
	and
	\[
	H_k\cap \Pc^{h_0}(\C)=\Ac_{s+ k,h_k}(\C_+)\cap \Pc^{h_0}(\C)=\Pc^{h_k}(\C),
	\] 
	since  $h_k\meg h_0$, $D^{h_k}\colon \Ac_{s+ k,h_k}(\C_+)\to \Ac_{s+k+2 h_k}(\C_+)$ is onto, and $D^{h_0-h_k}\colon \Ac_{s+k+2 h_k}(\C_+)\to \Ac_{s+k+2 h_0}(\C_+)$ is an isomorphism, thanks to Propositions~\ref{prop:23} and~\ref{prop:33}. Thus, $H\cap \ker \partial_2^{h_0}=V$, so that $H\subseteq\Ac_{s,h_0}(\Uc)$ with an equivalent (resp.\ proportional) seminorm.
	
	It then remains to consider the case in which
        $K_3=\emptyset$. Since $H$ is not trivial, $K_2$ must be
        non-empty, so that $s  \in - \N$. In addition, the preceding discussion shows that $H$ contains the orthogonal direct sum $\bigoplus_{k\in K_2} [\Pc_k(E)\otimes_2 H_k]$, and that this subspace is complete and dense (but not $\Aff(\Uc)$-invariant, in general). The proof is therefore complete.  
\end{proof}

\begin{teo}\label{teo:4}
	Take $s\in \R$. If $s\in -\N$, then define
	\[
	\widetilde \Ac_{s,1-s}\coloneqq \Set{f\in \Ac_{s,1-s}\colon \forall k\in \N\:\: \pi_k f\in \Pc_k(\C^n)\otimes \Ac_{s+k,(1-s-k)_+}(\C_+)},
	\]
	where $(\pi_k f)(\zeta,z)=\frac {1}{k!} \partial_1^k f(0,z) \zeta^k$ for every $f\in \Hol(\Uc)$, for every $(\zeta,z)\in \Uc$, and for every $k\in\N$. Then, $\widetilde \Ac_{s,1-s}$ is $\widetilde \Aut$-$\widetilde U_s$-invariant with its seminorm.
	
	Conversely, let $H$ be a non-trivial strongly decent and saturated  semi-Hilbert space of holomorphic functions such that $\widetilde U_s$ induces a bounded (resp.\ isometric) representation of  $\widetilde \Aut$ in $H$. Then, either $s\Meg 0$ and $H=\Ac_s$, or $s\in -\N$ and $H=\widetilde \Ac_{s,1-s}$ with an equivalent (resp.\ proportional) seminorm.
\end{teo}

See~\cite[Theorem 5.5]{Arcozzietal} for a different description of $\widetilde \Ac_{0,1}$. 

\begin{proof}
	\textsc{Step I.} By the description of $GL(\Uc)$ provided in Subsection~\ref{sec:2:1}, it is clear that $	\widetilde \Ac_{s,1-s}$ is $GL(\Uc)$-$U_s$-invariant with its seminorm. Since $\Aff(\Uc)$ is the semi-direct product of $\Hb_n$ and $GL(\Uc)$ (cf.~Subsection~\ref{sec:2:1}), it will then suffice to prove that $	\widetilde \Ac_{s,1-s}$ is $\Hb_n$-invariant (the seminorm is then necessarily $\Hb_n$-invariant) and also $\widetilde U_\lambda(\iota)$-invariant with its seminorm, where $\iota\in \widetilde \Aut$ is such that
	\[
	\iota\colon \Uc\ni (\zeta,z)\mapsto \Big(-\frac{i\zeta}{z},-\frac{1}{z}\Big)\in \Uc.
	\]
	
	Observe first that the proof of Theorem~\ref{teo:3} shows that $\pi_k(\Ac_{s,1-s})=\Pc_k(\C^n)\otimes_2 \Ac_{s+k,1-s}(\C_+)$ with proportional seminorms, so that $\Pc_k(\C^n)\otimes_2 \Ac_{s+k,1-s}(\C_+)$ is a subspace of $ \Ac_{s,1-s}$, and carries a proportional seminorm. 
	In addition, observe that, in order to prove $\Hb_n$-invariance, it will suffice to prove that $U_s(\Hb_n) \pi_k(\widetilde \Ac_{s,1-s})\subseteq \widetilde \Ac_{s,1-s} $ for every $k\in \N$. Therefore, it will suffice to prove that $U_s(\Hb_n) (p\otimes f)\in \widetilde \Ac_{s,1-s}$ for every $p\in \Pc_k(\C^n)$ and for every $f\in  \Ac_{s+k, (1-s-k)_+}(\C_+)$.
	Then, take $(\zeta,x)\in \Hb_n$ with $\zeta\neq 0$, and observe that
	\[
	\begin{split}
		\pi_h(U_s (\zeta,x)(p\otimes f))(\zeta',z')&=\sum_{h'=(h-k)_+}^h\frac{1}{h'! (h-h')!} (2 i \langle\zeta'\vert\zeta\rangle)^{h'}\partial_{\zeta'}^{h-h'} p(\zeta) f^{(h')}(z'+x+i\abs{\zeta}^2)
	\end{split}
	\]
	for every $h\in \N$ and for every $(\zeta',z')\in \Uc$, so that it will suffice to prove that $f^{(h')}(\,\cdot\,+x+i\abs{\zeta}^2)\in \Ac_{s+h, (1-s-h)_+}(\C_+)$ for every $h\in \N$ and for every $h'=(h-k)_+,\dots,k$. Since $h'+(1-s-h)_+\Meg (1-s-k)_+$ and $s+ k+2[h'+ (1-s-h)_+]\Meg s+2(1-s-k)_+$, the assertion follows from Lemma~\ref{lem:20}. The case $\zeta=0$ is clear.
	
	Then, using~\eqref{eq:7}  and Proposition~\ref{prop:29}, we see that $\widetilde \Ac_{s,1-s}$ is $\widetilde U_s(\iota)$-invariant with its seminorm. Thus, $\widetilde \Ac_{s,1-s}$ is $\widetilde \Aut$-$\widetilde U_s$-invariant with its seminorm.
	
	\textsc{Step II.} Take $H$ as in the statement. 
	Denote by $V$ the closure of $\Set{0}$ in $H$, so that $V$ is closed in $\Hol(\Uc)$ and the canonical mapping $H\to \Hol(\Uc)/V$ is continuous (and non-trivial). 
	If $V=\Set{0}$, then Theorem~\ref{teo:1} shows that $s\Meg 0$ and that $H=\Ac_{s}$. If, otherwise, $V\neq \Set{0}$, then Proposition~\ref{prop:30} shows that $s\in -\N$ and that $V=\Pc^{1-s}(\C^{n+1})=\bigoplus_{k\meg -s} [\Pc_k(\C^n)\otimes \Pc^{1-s-k}(\C)]$. In addition, Proposition~\ref{prop:30} again implies that $H$ is dense in $\Hol(\Uc)$. Therefore, arguing as in the proof of Theorem~\ref{teo:3}, we see that $\pi_k(H)=\Pc_k(\C^n)\otimes \Ac_{s+k,(1+s+k)_+}(\C)$ for every $k\in \N$. In particular, $H\cap \ker\partial_2^{1-s}=V$, so that Corollary~\ref{cor:1} implies that $H\subseteq \Ac_{s,1-s}$ with an equivalent (resp.\ proportional) seminorm. It then follows that $H=\widetilde \Ac_{s,1-s}$.
\end{proof}

\begin{oss}\label{oss:3}
	Take $s\in -\N$ and $s'>0$, and assume that $n>0$. Then, the unitary representations $\widetilde U_s$ and $\widetilde U_{s'}$ of $\widetilde \Aut(\Uc)$ in $\widehat \Ac_{s,1-s}$, identified with the Hausdorff space associated with $\widetilde\Ac_{s,1-s}$, and in $\Ac_{s'}$, respectively, are \emph{not} (unitarily) equivalent.
\end{oss}

Notice that this contradicts~\cite[Theorem 5.4]{FarautKoranyi} when $n>0$.

\begin{proof}
	It suffices to observe that, if $\widetilde K$ denotes the stabilizer of $(0,i)$ in $\widetilde \Aut(\Uc)$, then $\Ac_{s'}$ contains a one-dimensional $\widetilde K$-$U_{s'}$-invariant subspace (namely, the one generated by $B^{-s'}_{(0,i)}$, cf.~Lemma~\ref{invariance:lem}), whereas the Hausdorff space associated with $\widetilde\Ac_{s,1-s}$ does not (cf.~Propositions~\ref{prop:38} and~\ref{prop:39})
\end{proof}

\section{Invariant Spaces on $B_{n+1}$}\label{sec:4}

We now indicate how the preceding results may be transferred to the bounded realization of $\Uc_{n+1}$, namely the unit ball $B_{n+1}$ in $\C^{n+1}$. Recall that the Cayley transform
\[
\Cc\colon B_{n+1}\ni (\zeta,z)\mapsto \left( \frac{\zeta}{1-z}, i \frac{1+z}{1-z} \right)\in \Uc_{n+1},
\]
with inverse
\[
\Cc^{-1} \colon \Uc_{n+1}\ni (\zeta,z)\mapsto \left(2i \frac{\zeta}{z+i} , \frac{z-i}{z+i}\right)\in B_{n+1},
\]
induces a birational biholomorphism of $B_{n+1}$ onto
$\Uc_{n+1}$. Then, $ \Aut (B_{n+1})=\Cc^{-1} \Aut(\Uc_{n+1})\Cc $ is
the group of biholomorphisms of $B_{n+1}$, and $\Cc^{-1} \Aff(\Uc)
\Cc$ is the stabilizer of $(0,1)$ in $\Aut(B_{n+1})$. Since this
latter group does not seem to be of particular significance in this
context, we shall consider only $\Aut(B_{n+1})$. Observe that, by an
abuse of notation, we may then identify $\widetilde \Aut(B_{n+1})$ with
$\Cc^{-1} \widetilde \Aut(\Uc_{n+1}
) \Cc$. 

In addition, observe that the complex Jacobian of $\Cc$ and $\Cc^{-1}$ are
\[
(J \Cc)(\zeta,z)=\frac{2 i}{(1-z)^{n+2}} \qquad \text{and} \qquad (J\Cc^{-1})(\zeta,z)=\frac{(2i)^{n+1}}{(z+i)^{n+2}}.
\]
Then, we define, for every $s\in \R$,
\[
(\Cc_s f)(\zeta,z)\coloneqq f(\Cc(\zeta,z)) \frac{(2i)^{s/(n+2)}}{(1-z)^{s}}
\]
for every $f\in\Hol(\Uc_{n+1})$ and for every $(\zeta,z)\in B_{n+1}$, so that
\[
(\Cc^{-1}_s f)(\zeta,z)=f(\Cc^{-1}(\zeta,z)) \frac{(2i)^{s(n+1)/(n+2)}}{(z+i)^s},
\]
for every $f\in \Hol(B_{n+1})$ and for every $(\zeta,z)\in \Uc_{n+1}$, where $(2i)^{s'}\coloneqq 2^{s'} \ee^{s'\pi i/2}$ for every $s'\in \R$.
If we define
\[
\widetilde \Us_s(\phi)\coloneqq \Cc_s \widetilde U_s(\Cc \phi\Cc^{-1})\Cc^{-1}_s
\]
for every $\phi\in \widetilde \Aut(B_{n+1})$ (with the above abuse of notation), then
\[
\widetilde \Us_s(\phi) f = (f\circ \phi^{-1}) (J \phi)^{s/(n+2)}
\]
for every $\phi\in \widetilde \Aut(B_{n+1})$ and for every $f\in \Hol(B_{n+1})$, for a suitable deifnition of $(J\phi)^{s/(n+2)}$. 

Thus, the $\Cc_s\Ac_s$, for $s\Meg 0$, are the $\widetilde \Aut(B_{n+1})$-$\widetilde \Us_s$-invariant RKHS in $\Hol(B_{n+1})$. 
Their reproducing kernel are given by
\[
[(\Cc_s\otimes \overline{\Cc_s}) B^{-s}](w,w')=2^{2s/(n+2)}(1-\langle w\vert w'\rangle)^{-s}
\]
for every $w,w'\in B_{n+1}$.

\begin{prop}
	If $s>n+1$, then 
	\[
	\Cc_s \Ac_s=\Set{f\in \Hol(B_{n+1})\colon  c'_s\int_{B_{n+1}} \abs{f(w)}^2 (1-\abs{w}^2)^{s-n-2} \,\dd w<\infty },
	\]
	endowed with the corresponding  norm, where $c'_s\coloneqq 2^{-2s/(n+2)} \frac{(s-1)(s-2)\cdots(s-n-1)}{ \pi^{n+1}}$.
\end{prop}

The proof is a tedious but simple computation which is left to the reader.

In order to provide a description of the remaining spaces, we shall consider the stabilizer $\Kc$ of $0$ in $\Aut(B_{n+1})$, which corresponds to the stabilizer $K$ of $(0,i)$ in $\Aut(\Uc_{n+1})$ by means of the Cayley transform. It is well known that $\Kc$ is the unitary group $U(n+1)$ on $\C^{n+1}$ (cf., e.g.,~\cite[Section 2.1]{Knapp}). Observe that, if we define
\[
\Us(\phi)f\coloneqq f\circ \phi^{-1}
\]
for every $\phi\in \Kc$ and for every $f\in \Hol(B_{n+1})$, then  $\Us$ is a continuous representation of $\Kc$ in $\Hol(B_{n+1})$. If we denote by $\widetilde \Kc$ the pre-image of $\Kc$ in $\widetilde \Aut(B_{n+1})$, then there is a character $\chi_s$ on $\widetilde \Kc$ (namely, $\phi\mapsto (J\phi)^{-s/(n+2)}(0)$) such that $\chi_s \Us_s$ is the representation of $\Kc$ in $\Hol(B_{n+1})$ induced by $\Us$. Consequently, a vector subspace of $\Hol(B_{n+1})$ is $\Kc$-$\Us$-invariant  if and only if it is $\widetilde \Kc$-$\Us_s$-invariant for every $s\in \R$. Analogously, $f\in \Hol(\Uc_{n+1})$ is a finite $\Kc$-$\Us$-vector, that is, generates a finite-dimensional $\Kc$-$\Us$-invariant vector subspace, if and only if it is a finite $\widetilde \Kc$-$\widetilde \Us_s$-vector for every $s\in \R$.

\begin{prop}\label{prop:38}
	The following hold:
	\begin{enumerate}
		\item[\textnormal{(1)}] the space $\Pc$ of holomorphic polynomials is the space of finite $\Kc$-$\Us$-vectors (or, equivalently, finite $\T$-$\Us$-vectors) in $\Hol(B_{n+1})$;
		
		\item[\textnormal{(2)}] for every $k\in\N$, the space $\Pc_k$ of homogeneous holomorphic polynomials of degree $k$ is $\Kc$-$\Us$-invariant and irreducible;
		
		\item[\textnormal{(3)}] if $X$ is a $\T$-$\Us$-invariant vector subspace of $\Hol(B_{n+1}) $ such that $\Us(f) X\subseteq X$ for every $f\in L^1(\T)$,\footnote{This happens, for example, if $\Us$ induces a continuous representation of $\T$ in $X$.} then $\Pc\cap \overline X=\Pc \cap X$, where $\overline X$ denotes the closure of $X$ in $\Hol(B_{n+1})$;
		
		\item[\textnormal{(4)}]   if $X$ is a strongly decent and saturated semi-Banach space of holomorphic functions on $B_{n+1}$ such that $\Us$ induces a \emph{continuous} representation of $\T$ in $X$, then $X\cap \Pc$ is dense in $X$;  
		
		\item[\textnormal{(5)}]   if $X$ is a semi-Banach space of holomorphic functions on $B_{n+1}$ such that  $\T$ induces a \emph{continuous} representation of $\T$ in $X$ and such that $\widetilde \Us_s(\phi)$ induces a continuous automorphism of $X$ for every $\phi\in \widetilde \Aut(B_{n+1})$, then $X$ is strongly decent and saturated if and only if for every $f\in L^1(\T)$, defining $\Us(f)$ as an endomorphism of $\Hol(B_{n+1})$, one has $\Us(f)X\subseteq X$ and  $\langle x', \Us(f) x\rangle=\int_\T f(\alpha)\langle x',\Us(\alpha)x\rangle\,\dd \alpha$ for every $x\in X$ and for every $x'\in X'$.
	\end{enumerate}
\end{prop}

Note that (1) and (2) are well known, whereas (3) and (4) are essentially based on~\cite[Proposition 1]{ArazyFisher2}. In (5) we essentially compare our strong decency (and saturation) conditions with the weak integrability conditions considered, for example, in~\cite[Theorems 5.1, 5.2, and 5.3]{Arazy}.\footnote{Notice that, in the cited references, there is no mention to the continuity of the representation of $\T$ in $X$ induced by $\Us$. Nonetheless, the fundamental~\cite[Proposition 1]{ArazyFisher2} requires this assumption (at least, its proof does), so that this omission should be considered as a minor mistake. In addition, if one assumes that $X$ is a \emph{separable} semi-\emph{Hilbert} space, then the continuity follows from the measurability conditions which are needed to define the above integrals, thanks to~\cite[Corollary 2 to Proposition 18 os Chapter VIII, \S\ 4, No.\ 6]{BourbakiInt2}. }  Notice that, for symmetric domains of higher rank, the proper closed invariant subspaces of the space of holomorphic functions need not be finite-dimensional, so that it does not seem possible to establish a precise equivalence between these two notions.

Before passing to the proof, we translate Proposition~\ref{prop:30} in this context. The proof is a simple verification and is omitted.

\begin{oss}\label{oss:1}
	Let $V$ be a \emph{proper} closed vector subspace of $\Hol(B_{n+1})$, and take $s\in \R$. Then, $V$ is $\widetilde \Aut(B_{n+1})$-$\widetilde \Us_s$-invariant if and only if $V=\Set{0}$ or $s\in -\N$ and $V$ is the space $\Pc^{1-s}$ of holomorphic polynomials of degree $<1-s$ on $\C^{n+1}$.
\end{oss}

\begin{proof}[Proof of Proposition~\ref{prop:38}.]
	(1) It is clear that every $p\in\Pc$ is a finite $\Kc$-$\Us$-vector. Conversely, if $f$ is a finite $\T$-$\Us$-vector in $\Hol(B_{n+1})$, then 
	\[
	\pi_k f\coloneqq \frac{1}{k!}f^{(k)}(0)(\,\cdot\,)^k=\int_{\T} \alpha^{k} \Us(\alpha) f\,\dd \alpha
	\]
	belongs to the (finite-dimensional, hence closed) $\T$-$\Us$-invariant vector subspace of $\Hol(B_{n+1})$ generated by $f$ for every $k\in\N$. Hence, only finitely many of the $\pi_k f$ may be non-zero, so that $f$ is a polynomial.
	
	(2) This is a consequence of~\cite[Theorem 2.1]{FarautKoranyi2}.
	
	(3)   Define $\pi_k$ as in (1), and observe that $\pi_k (X)\subseteq X$ by the assumption, so that $\pi_k(X)=\Pc_k\cap X$. Then, take $f\in \overline X$, and let $(f_j)$ be a sequence in $X$ which converges to $f$ in $\Hol(B_{n+1})$. Then, $\pi_k(f_j)$ converges to $\pi_k(f)$ in $\Hol(B_{n+1})$. Since the $\pi_k(f_j)$ belong to the (finite-dimensional, hence) closed vector subspace $\pi_k(X)$ of $\Hol(B_{n+1})$, this proves that $\pi_k(f)\in \pi_k(X)$. Thus, $\pi_k(X)=\pi_k(\overline X)$, that is, $X\cap \Pc_k=\overline X\cap \Pc_k$, for every $k\in\N$.
	Thus, $X\cap \Pc=\overline X\cap \Pc$.  
	
	(4) By Proposition~\ref{prop:31}, if $V$ is the closure of $\Set{0}$ in $X$, then $V$ is closed in $\Hol(B_{n+1})$ and  the canonical mapping $X\to \Hol(B_{n+1})/V$ is continuous.  For every $k\in\N$,   define $\pi'_k\coloneqq \int_{\T} \alpha^{k} \Us(\alpha) \,\dd \alpha$ as a continuous linear mapping $X\to X/V$, so that, by the continuity of the linear mapping $X\to \Hol(B_{n+1})/V$, 
	\[
	\pi_k(f)\in \pi'_k(f) 
	\]
	for every $f\in X$, where $\pi_k$ is defined as in (1). In particular, $\pi'_k(X)\subseteq (X\cap \Pc+V)/V$. Then, observe that the properties of the Fejér kernel show that
	\[
	 \sum_{k=0}^m \left(1-\frac{k}{m+1}\right) \pi'_k(f)=\int_\T \Us(\alpha)f  \sum_{k=-m}^m \left(1-\frac{\abs{k}}{m+1}\right)\alpha^k \,\dd \alpha
	\] 
	converges to $f$ in the Banach space $X/V$, whence the result.
	
	(5)  Assume that $\Us$ induces a continuous representation of $\T$ in $X$ and that  $\Us(f)X\subseteq X$ and
	\[
	\langle x', \Us(f) x\rangle=\int_\T f(\alpha)\langle x',\Us(\alpha)x\rangle\,\dd \alpha
	\]  
	for every $f\in L^1(\T)$. In particular, $\Us(f)$ induces a continuous automorphism of $X$ for every $f\in L^1(\T)$.
	
	Observe that the closure $V$ of $\Set{0}$ in $X$ is  $\widetilde \Aut(B_{n+1})$-$\widetilde \Us_s$-invariant, so that Remark~\ref{oss:1} implies that either $V=\Set{0}$, or $s\in-\N$ and $V=\Pc^{1-s}$, or $V$ is dense in $\Hol(B_{n+1})$. Define $\pi_k$, for every $k\in \N$, as in (1), so that $\pi_k$ induces a continuous endomorphism of $X$ by the previous remarks.
	Observe that  (3) implies that $\Pc \cap V=\Pc\cap \overline V$, where $\overline V$ denotes the closure of $V$ in $\Hol(B_{n+1})$. In addition, arguing as in the proof of (4), we see that $\Pc\cap X$ is dense in $X$.	
	Since $X\neq V$, this proves that $V$ is not dense in $\Hol(B_{n+1})$. 
	
	If  $V=\Set{0}$, then the continuity of $\pi_0 $ on $X$ implies that the mapping $f \mapsto f(0)$ is continuous on $X$. Since $\widetilde \Us_s(\phi)$ is continuous on $X$ for every $\phi\in \widetilde H(B_{n+1})$ and since $\widetilde \Aut(B_{n+1})$ acts transitively on $B_{n+1}$, this implies that the mapping $L_w f\mapsto f(w)$ is continuous on $X$ for every $w\in B_{n+1}$. Since the common kernel of the $L_w$ in $\Hol(B_{n+1})$ is $\Set{0}=V$, it is clear that $X$ is strongly decent and saturated. 
	
	If, otherwise, $s\in -\N$ and $V=\Pc^{1-s}$, then the mapping
	\[
	L_{0,v}\colon X\ni f \mapsto \partial^{1-s}_v f(0)= \partial^{1-s}_v(\pi_{1-s} f)(0)\in \C
	\]
	is continuous for every $v\in \C^{n+1}$. Arguing as before, this implies that the mapping $L_{w,v}\colon X\ni f\mapsto \partial^{1-s}_v f(w)\in \C$ is continuous for every $w\in B_{n+1}$ and for every $v\in \C^{n+1}$.\footnote{  In general, one obtains that the mapping $X\ni f \mapsto Y_{\phi,v} f \in \C$ is continuous, where $Y_{\phi,v}\colon f \mapsto \partial^{1-s}_v [\widetilde U_s(\phi^{-1}) f](0)$, with $\phi\in \widetilde \Aut(B_{n+1})$ and $v\in \C^{n+1}$. It is clear that $Y_{\phi,v}$ is a differential operator at $w$, and  that its highest order term is $ (J\phi^{-1})^{s/(n+2)}(0) \partial_{(\phi^{-1})'(0)v }^{1-s} $. Since $Y_{\phi,v}$ induces a continuous linear functional on $X$, it must vanish on $\Pc^{1-s}$, so that it cannot have lower order terms. } Since the common kernel of these mappings is $\Pc^{1-s}=V$, $X$ is strongly decent and saturated also in this case.	
	
	The converse implication follows from   the continuity of the canonical mapping $X\to \Hol(\Uc_{n+1})/V$.
\end{proof}

We may now describe the spaces $\Cc_s \Ac_s$, for $s\Meg 0$, and $\Cc_s \widetilde \Ac_{s,1-s}$, for $s\in -\N$, in more precise, though quite abstract, way. We refer the reader to~\cite{Peloso,Yan} for other descriptions in terms of integro-differential seminorms.
We endow $\Pc$ with the Fisher inner product
\[
\langle p\vert q\rangle_\Fc\coloneqq p(\nabla) q^*,
\]
where $q^*\colon w \mapsto \overline{q(\overline w)}$, for ever $p,q\in \Pc$. Equivalently (cf.~\cite[Proposition XI.1.1]{FarautKoranyi}),
\[
\langle p\vert q\rangle_\Fc=\frac{1}{\pi^{n+1}}\int_{\C^{n+1}} p(w)\overline{q(w)}\ee^{-\abs{w}^2}\,\dd w
\]
for every $p,q\in \Pc$.

\begin{prop}\label{prop:39}
	Take $s\Meg 0$. Then, $\Cc_s \Ac_s$ is the space of $f=\sum_k f_k\in \Hol(B_{n+1})$, where $f_k$ is the homogeneous component of degree $k$ of $f$, such that
	\[
	\norm{f}_{\Cc_s \Ac_s}^2=  2^{-2 s/(n+2)}\sum_{k\in \N} \frac{1}{(s)^k} \norm{f_k}_\Fc^2<\infty,
	\]
	where $(s)^k=s(s+1)\cdots(s+k-1)$ for every $k\in\N$.\footnote{  When $s=0$, this means that $f_k=0$ for every $k\Meg 1$. }
	
	Take $s\in -\N$. Then, $\Cc_s \widetilde \Ac_{s,1-s}$ is the space of $f=\sum_k f_k\in \Hol(B_{n+1})$ such that
	\[
	  \norm{f}^2_{\Cc_s \widetilde \Ac_{s,1-s}}=2^{-2 (2-s)/(n+2)}\sum_{k\Meg 1-s} \frac{1}{(-s)!(k+s-1)!}\norm{f_k}_\Fc^2<\infty. 
	\]
\end{prop}

\begin{proof}
	  For the first assertion, one may either compute the transferred  norm, or use~\cite[Corollary 3.7]{FarautKoranyi2}. For the second assertion, it is sufficient to compute the transferred seminom of $z^k$, $k\in\N$, using a suitable generalization of~\cite[Remark 4.5]{Garrigos}, whose proof is analogous to that of Proposition~\ref{prop:29}. 
\end{proof}

\section{Appendix: Mean-Periodic Functions}\label{sec:app}

  In this appendix, for future reference, we let $D$ be a general Siegel domain, so that
\[
D=\Set{(\zeta,z)\in \C^n\times \C^m\colon \Im z-\Phi(\zeta)\in \Omega},
\]
where $\Omega\subseteq \R^m$ is an open convex cone not containing affine lines ($\R_+^*$ in the cases considered in this paper) and $\Phi\colon \C^n\to \C^m$ is a non-degenerate hermitian quadratic form such that $\Phi(\zeta)\in \overline\Omega$ for every $\zeta\in \C^n$ (the scalar product on $\C^n$ in the cases considered in this paper). In this case, we then set $\rho(\zeta,z)=\Im z-\Phi(\zeta)$ for every $(\zeta,z)\in D$. 
We shall define $\Nc\coloneqq \C^n\times \R^m$, endowed with the group structure given by
\[
(\zeta,x)(\zeta',x')=(\zeta+\zeta', x+x'+2 \Im \Phi(\zeta,\zeta'))
\]
for every $(\zeta,x),(\zeta',x')\in \Nc$.   If we endow $\C^n\times C^m$ with the product  
\[
(\zeta,z)\cdot (\zeta',z')=(\zeta+\zeta', z+z'+2 i \Phi(\zeta',\zeta))
\]
for every $(\zeta,z),(\zeta',z')\in \C^n\times \C^m$, then $D$ is a semigroup of $\C^n\times \C^m$ and $\Nc$ may be identified with the subgroup $\rho^{-1}(0)$ of $\C^n\times \C^m$ by means of the isomorphism $(\zeta,x)\mapsto (\zeta,x+i\Phi(\zeta))$. Then, $\Nc$ acts on $D$ by affine biholomorphisms.  
We shall denote by $G_\delta$ the group of affine automorphisms of $D$ which is the semi-direct product of $\Nc$ and the group of the dilations $R\cdot (\zeta,z)=(R\zeta, R^2z)$.

We now present some results which may be interpreted in the spirit of the study of mean-periodic functions (cf., e.g.,~\cite{Schwartz}). We shall actually give a reasonable description of all mean-periodic functions in $\Hol(D)$ with respect to the group $G_\delta$, acting by composition.

\begin{prop}\label{prop:6}
	Let $V$ be a closed $G_\delta$-invariant subspace of $\Hol(D)$. Then, $V=\overline{\Pc\cap V}$, where $\Pc$ denotes the space of (holomorphic) polynomials on $D$ and $\overline{\Pc\cap V}$ denotes the closure of $\Pc\cap V$ in $\Hol(D)$.
\end{prop}

\begin{proof}
	Consider the duality between $\Hol(D)$ and the
        space  $\Ec'_0(D)$ of compactly supported Radon measures on $D$, induced by
        the natural inclusion of $\Hol(D)$ in the space $ C(D)$ of
        continuous functions on $D$ and the natural duality between
        $C(D)$ and $\Ec'_0(D)$. Observe that a vector subspace of
        $\Hol(D)$ is closed if and only if it is weakly closed with
        respect to this duality, so that, in particular,
        $V=V^{\circ\circ}$, where $V^\circ=\Set{\mi \in
          \Ec'_0(D)\colon\forall f\in V\:\: \langle \mi, f \rangle=0 }$ denotes the polar  of $V$, and $V^{\circ\circ}$ its bipolar.  
	Then, take $\mi\in V^\circ$, and consider, for every $f\in V$, the function
	\[
	F_{f,\mi}\colon D-(0,i v)\ni (\zeta,z)\mapsto \int_{\Supp \mi} f((\zeta,z)\cdot (\zeta',z'))\,\dd \mi(\zeta',z'), 
	\]
	where $v\in \Omega$ is such that $\Supp \mi\subseteq (0,i v)+ D$. Observe that clearly $F_{f,\mi}$ is well defined and holomorphic in the second variable. It is actually real analytic, but we shall not need this fact.

	Then, observe that $F_{f,\mi}(\zeta,x+i\Phi(\zeta))=0$ for every $(\zeta,x)\in \Nc$, so that, by the holomorphy in the second variable, $F_{f,\mi}(\zeta,z)=0$ for every $(\zeta,z)\in D-(0,i v)$. Consequently, the arbitrariness of $\mi$ implies that $f((\zeta,z)\,\cdot\,\,)\in V$ for every $(\zeta,z)\in D$.
	
	Now, observe that 
	\[
	\frac{\dd^k}{\dd R^k} f((\zeta,z) \cdot [R\cdot (\zeta',z')]  )= \sum_{  k_1+2 k_2=k}\frac{k!}{k_1! k_2!} f^{(k_1+k_2)}((\zeta,z)[R\cdot (\zeta',z')]) \cdot (\zeta',2 R z')^{k_1} \cdot (0,z')^{k_2}
	\]
	for every $k\in \N$, for every $R>0$, and for every $(\zeta,z),(\zeta',z')\in D$, so that  passing to the limit for $R\to 0^+$, we see that the polynomial mapping
	\[
	P_{f,k,(\zeta,z)}\colon(\zeta',z')\mapsto \sum_{  k_1+2 k_2=k} \frac{1}{k_1! k_2!} f^{(k_1+k_2)}(\zeta,z) \cdot (\zeta',0)^{k_1} \cdot (0,z')^{k_2}= \sum_{  k_1+2 k_2=k} \frac{1}{k_1! k_2!} \partial_{(\zeta',0)}^{k_1} \partial_{(0,z')}^{k_2} f(\zeta,z) 
	\]
	belongs to $V$ for every $k\in \N$ and for every $(\zeta,z)\in D$. Observe that, at least formally, 
	\[
	\sum_k P_{f,k,(\zeta,z)}=\sum_{k} \sum_{k_1+k_2=k}\frac{1}{k_1! k_2!} f^{(k)}(\zeta,z) \cdot (\zeta',0)^{k_1} \cdot (0,z')^{k_2}=\sum_k \frac{1}{k!} f^{(k)}(\zeta,z)\cdot (\zeta',z')^k,
	\]
	so that $f=\sum_k P_{f,k,(\zeta,z)}$ on every open set where the sum converges absolutely; in particular, on every open ball centred at $(\zeta,z)$ whose closure is contained in $D$. In a similar way, by translation one sees that the same holds on $(\zeta,z) \cdot B$, where $B$ is an open Euclidean ball centred at $(0,0)$ such that the closure of $(\zeta,z) \cdot B$ is contained in $D$. 
	
	Now, take $f\in V$ and $\mi\in (\Pc\cap V)^\circ$. In addition, take $R>0$ so that $ \Supp{\mi}$ is contained in the Euclidean ball $B_{E\times F_\C}((0,0),R)$ of centre $(0,0)$ and radius $R$,  and observe that, if $z\in D$ and $\rho(\zeta,z)$ is sufficiently large, then (the closure of) $(\zeta,z)\cdot B_{E\times F_\C}((0,0),R)$ is contained in $D$. 
	Observe that $\int P_{f,k,(\zeta,z)}\,\dd \mi=0$ for every $k\in \N$ and for every $(\zeta,z)\in D$, thanks to the preceding remarks. Since $P_{f,k,(\zeta,z)}=P_{f((\zeta,z)\,\cdot\,),k,(0,0)}$ and the series $\sum_{k} P_{f((\zeta,z)\,\cdot\,),k,(0,0)}$ converges  uniformly to $f((\zeta,z)\,\cdot\,)$ on $(\zeta,z)\cdot B_{E\times F_\C}((0,0),R)$ (if $\rho(\zeta,z)$ is sufficiently large) by the preceding remarks, we conclude that $F_{f,\mi}(\zeta,z)=0$ if $\rho(\zeta,z)$ is sufficiently large. Since $F_{f,\mi}$ is holomorphic in the second variable, this is sufficient to conclude that $F_{f,\mi}=0$. In particular, $F_{f,\mi}(0,0)=0$, that is, $\int f\,\dd \mi=0$.
	Therefore, the arbitrariness of $\mi$ implies that $f\in \overline{\Pc\cap V}$. The arbitrariness of $f$ then shows that $V=\overline{\Pc\cap V}$.
\end{proof}

\begin{cor}\label{lem:2}
	The space $\Pc$ of holomorphic polynomials on $D$ is dense in $\Hol(D)$.
\end{cor}

\begin{cor}\label{cor:7}
	Let $V$ be a  closed $G_\delta$-invariant subspace of
        $\Hol(D)$. Then, there is a set $\Ic$ of homogeneous\footnote{With respect to the dilations $R\cdot (\zeta,x)=(R\zeta, R^2 x)$.} distributions on $\Nc$ supported in $\Set{(0,0)}$ such that 
	\[
	V=\Set{f\in \Hol(D)\colon f*\Ic=0}.
	\]
\end{cor}

\begin{proof}
	Observe that, if $\Ic$ is a set of homogeneous distributions on $\Nc$, supported in $\Set{(0,0)}$, then $V_\Ic\coloneqq \Set{f\in \Hol(D)\colon f*\Ic=0}$ is a closed $G_\delta$-invariant subspace of $\Hol(D)$. By Proposition~\ref{prop:6}, we only need to show that $\Pc\cap V=\Pc\cap V_\Ic $ for a suitable $\Ic$.
	
	For every $k\in \N$, denote by $\Pc_k$ the space of holomorphic polynomials on $E\times F_\C$ which are homogeneous of degree $k$ with respect to the dilations $R\cdot (\zeta,z)=(R\zeta, R^2 z)$, and observe that $\Pc\cap V=\bigoplus_{k} \Pc_k\cap V $ since $\Pc\cap V$ is dilation-invariant. 
	Observe that, if $\Ic_k$ denotes the space of homogeneous distributions on $\Nc$ of degree $k$ supported in $\Set{(0,0)}$, then $\langle \Ic_k, \Pc_{k'}\rangle=0$ if $k\neq k'$, by homogeneity, and that the canonical pairing between $\Pc_k$ and $\Ic_k$ induces an isomorphism of $\Pc_k'$ with a quotient of $\Ic_k$ (with $\Ic_k$, if $n=0$). 
	Consequently, for every $k\in \N$, we may find a subset $\Ic'_k$ of $\Ic_k$ such that $\Pc_k\cap V=\Pc_k\cap \Ic_k'^\circ $. 
	Hence, $\Pc\cap V=\bigcap_{k\in \N} \Ic_k'^\circ=(\bigcup_{k\in \N} \Ic_k')^\circ$. 
	Then, set $\Ic\coloneqq \bigcup_{k\in\N} \check\Ic'_k$, where $\check\,$ denotes the action of the inversion $(\zeta,x)\mapsto (\zeta,x)^{-1}=(-\zeta,-x)$, and let us prove that $\Pc\cap V=\Pc\cap V_{\Ic'}$. 
	Take $P\in \Pc\cap V$, and observe that, since $\Pc\cap V$ is invariant under the action of $\Nc$, also $P((\zeta,z)\,\cdot\,)\in \Pc\cap V$ for every $(\zeta,z)\in E\times F_\C$, with the notation of the proof of Proposition~\ref{prop:6} (cf.~the proof of Proposition~\ref{prop:6} or argue directly using Taylor's formula), so that 
	\[
	(P*I)(\zeta,z)=\langle \check I, P((\zeta,z)\,\cdot\,)\rangle=0
	\]
	for every $I\in \Ic$, so that $P\in V_\Ic$ by the arbitrariness of $(\zeta,z)\in D$ and $I\in \Ic$. Conversely, take $P\in V_\Ic$. Then, for every $I\in \Ic$,
	\[
	\langle \check I, P\rangle=(P*I)(0,0)=0
	\]
	by continuity (or holomorphy), 	so that $P\in \Pc\cap V$.
\end{proof}

\end{document}